\documentclass[a4paper,11pt,reqno]{customart}

\usepackage[utf8x]{inputenc}
\usepackage[margin=1in]{geometry}
\usepackage{verbatim}
\usepackage{color}
\usepackage[table]{xcolor}
\usepackage{listings}
\usepackage{amssymb,amsfonts,amsthm,amsmath}
\usepackage{url}
\usepackage{enumerate}
\usepackage{todonotes}
\usepackage[all]{xy}
\usepackage{multirow}
\usepackage{attachfile}
\usepackage{stmaryrd}
\usepackage{footnote}
\usepackage{epigraph}
\makesavenoteenv{tabular}
\usepackage{hyperref}

\makeatletter
\newcommand\footnoteref[1]{\protected@xdef\@thefnmark{\ref{#1}}\@footnotemark}
\makeatother

\hypersetup{
	pdfborder={0 0 0}
}

\definecolor{grey}{rgb}{0.95,0.95,0.95}
\definecolor{green}{rgb}{0.2,0.6,0.4}

\newcommand{\N}{\mathbb{N}}

\renewcommand{\O}{\textbf{0}}

\newcommand{\restr}{\upharpoonright}

\newcommand{\concat}{\frown}

\newcommand{\imp}{\rightarrow}

\newcommand{\biimp}{\leftrightarrow}

\newcommand{\Psf}{\mathsf{P}}
\newcommand{\Qsf}{\mathsf{Q}}

\newcommand{\Ccal}{\mathcal{C}}
\newcommand{\Dcal}{\mathcal{D}}
\newcommand{\Ecal}{\mathcal{E}}

\newcommand{\Qcal}{\mathcal{Q}}
\newcommand{\Rcal}{\mathcal{R}}

\newcommand{\uh}{{\upharpoonright}}

\renewcommand{\setminus}{\smallsetminus}

\newcommand{\set}[1]{\left\{ #1 \right\}}
\newcommand{\card}[1]{\left| #1 \right|}

\newcommand{\tuple}[1]{\left\langle #1 \right\rangle}

\newcommand{\cond}[1]{\left\{\begin{array}{ll} #1 \end{array}\right.}

\newcommand{\s}[1]{\ensuremath{\sf{#1}}}

\newcommand{\wwkls}[1]{#1\mbox{-}\s{WWKL}}

\DeclareMathOperator{\rca}{\s{RCA}_0}
\DeclareMathOperator{\aca}{\s{ACA}_0}
\DeclareMathOperator{\wkl}{\s{WKL}_0}
\DeclareMathOperator{\wwkl}{\s{WWKL}_0}
\DeclareMathOperator{\dnr}{\s{DNR}}

\DeclareMathOperator{\bst}{\s{B}\Sigma^0_2}

\DeclareMathOperator{\kl}{\s{KL}}
\DeclareMathOperator{\rt}{\s{RT}}
\DeclareMathOperator{\srt}{\s{SRT}}

\DeclareMathOperator{\rrt}{\s{RRT}}

\DeclareMathOperator{\ads}{\s{ADS}}
\DeclareMathOperator{\sads}{\s{SADS}}

\DeclareMathOperator{\cac}{\s{CAC}}

\DeclareMathOperator{\coh}{\s{COH}}
\DeclareMathOperator{\pizog}{\Pi^0_1\s{G}}

\DeclareMathOperator{\ts}{\s{TS}}
\DeclareMathOperator{\sts}{\s{STS}}

\DeclareMathOperator{\fs}{\s{FS}}
\DeclareMathOperator{\sfs}{\s{SFS}}
\DeclareMathOperator{\emo}{\s{EM}}

\newtheoremstyle{custom}
  {10pt}
  {10pt}
  {\normalfont}
  {}
  {\bfseries}
  {}
  { }
  {}

\theoremstyle{custom}

\usepackage{xcolor}	
\usepackage{soul}

\newtheorem{theorem}{Theorem}[section]
\newtheorem{lemma}[theorem]{Lemma}
\newtheorem{definition}[theorem]{Definition}

\newtheorem{question}[theorem]{Question}

\newtheorem{corollary}[theorem]{Corollary}

\begin{document}

\title[The weakness of being cohesive, thin or free in reverse mathematics]
	{The weakness of being cohesive, thin or free\\ in reverse mathematics}
\author{Ludovic Patey}
\email{ludovic.patey@computability.fr}

\maketitle

\begin{abstract}
Informally, a mathematical statement is \emph{robust} if its strength is left unchanged
under variations of the statement. 
In this paper, we investigate the lack of robustness of Ramsey's theorem and its consequence
under the frameworks of reverse mathematics and computable reducibility.
To this end, we study the degrees of unsolvability of cohesive sets
for different uniformly computable sequence of sets and identify different
layers of unsolvability. This analysis enables us to answer some questions of Wang
about how typical sets help computing cohesive sets.

We also study the impact of the number of colors in 
the computable reducibility between coloring statements.
In particular, we strengthen the proof by Dzhafarov that cohesiveness does not strongly reduce to stable
Ramsey's theorem for pairs, revealing the combinatorial nature of this non-reducibility
and prove that whenever $k$ is greater than $\ell$, stable Ramsey's theorem for $n$-tuples and $k$ colors
is not computably reducible to Ramsey's theorem for $n$-tuples and $\ell$ colors. In this sense, 
Ramsey's theorem is not robust with respect to his number of colors over computable reducibility.
Finally, we separate the thin set and free set theorem from Ramsey's theorem for pairs
and identify an infinite decreasing hierarchy of thin set theorems in reverse mathematics.
This shows that in reverse mathematics, the strength of Ramsey's theorem is very sensitive to the number of colors in the output set. 
In particular, it enables us to answer several related questions asked by Cholak, Giusto, Hirst and Jockusch.
\end{abstract}

\section{Introduction}

Ramsey's theorem ($\rt^n_k$) asserts that any $k$-coloring of $[\N]^n$
admits an infinite monochromatic set, where $[\N]^n$ stands for the $n$-tuples over~$\N$.
In this paper, we study the lack of \emph{robustness} of Ramsey's theorem and its consequences
under the frameworks of reverse mathematics and computable reducibility.
Informally, a mathematical statement is \emph{robust} within a given framework
if its strength is invariant under slight variations of the statement.
In reverse mathematics, robustness can be understood as equiprovability of the various statements over
the base theory $\rca$, while in computability, a $\Pi^1_2$ statement
is robust if its variations are computably equivalent. We shall detail further reverse mathematics
and computable reducibility in sections~\ref{subsect:intro-reverse-mathematics} and~\ref{subsect:intro-reducibilities}.
Our investigations follow three axes.

\emph{Axis 1:} We first study the degrees of unsolvability of cohesiveness.
Given a sequence of sets of integers $R_0, R_1, \dots$, an infinite set is \emph{$\vec{R}$-cohesive}
if it is almost included in $R_i$ or $\overline{R}_i$ for each~$i$.
$\coh$ is a consequence of Ramsey's theorem for pairs which finds many practical applications
in computability and reverse mathematics. Jockusch and Stephan~\cite{Jockusch1993cohesive} have studied
the computational strength of cohesive sets for maximally difficult sequences of sets.
We reveal that $\coh$ contains hierarchies of complexity of its instances
by establishing a one-to-one correspondence between instances of $\coh$
and instances of K\"onig's lemma ($\kl$). This shows that the strength of $\coh$
depends on the considered class of its instances, and is therefore not \emph{inner robust}.
This correspondance enables us to reprove the existence of a computable sequence of sets
with no low cohesive set~\cite{Jockusch1993cohesive} and to answer several questions asked by Wang~\cite{Wang2013Omitting}
about how randomness and genericity help in solving computably unsolvable instances of~$\coh$.

\emph{Axis 2:} A simple color amalgamation argument shows that $\rt^n_k$
and $\rt^n_\ell$ are provably equivalent in reverse mathematics whenever $k, \ell \geq 2$.
In this sense, Ramsey's theorem is robust with respect to numbers of colors in reverse mathematics.
However, the standard proof that $\rt^n_k \imp \rt^n_{k+1}$ involves two applications of $\rt^n_k$.
Mileti~\cite{Mileti2004Partition} first wondered whether those two applications were really necessary.
The question has been later formalized thanks to Weihrauch and computable reducibility
and investigated by Dorais, Dzhafarov, Hirst, Mileti and Shafer~\cite{Dorais2016uniform}, Hirschfeldt and Jockusch~\cite{Hirschfeldtnotions}, Brattka and Rakotoniaina~\cite{Brattka2015Uniform}, among others.
We answer positively by proving that for every~$n \geq 2$, $\rt^n_k$ is not computably reducible
to $\rt^n_\ell$ whenever $k > \ell \geq 2$. Therefore, Ramsey's theorem is not robust
with respect to the number of colors under computable reducibility.

\emph{Axis 3:} Last, we investigate the reverse mathematics of a weakening of Ramsey's theorem
in which more colors are allowed in the resulting set. The thin set theorem ($\ts^n_k$)
asserts that for any $k$-coloring of~$[\N]^n$, there is an infinite set~$H$ such that~$[H]^n$
avoids at least one color. We show that the thin set theorem is not robust in reverse mathematics
by proving that for every $n, m, \ell \geq 2$, $\ts^n_k$ does not imply $\ts^m_\ell$
for sufficiently large $k$'s. This is the first example of an infinite decreasing hiearchy
in reverse mathematics. This enables us to answer several questions from
Cholak, Giusto, Hirst and Jockusch~\cite{Cholak2001Free}, Mont\'alban~\cite{Montalban2011Open}
and Hirschfeldt~\cite{Hirschfeldt2015Slicing} about the strength of the thin set theorem
and its strengthening, the free set theorem, with respect to Ramsey's theorem for pairs.

\subsection{Reverse mathematics}\label{subsect:intro-reverse-mathematics}

Reverse mathematics is a vast mathematical program
whose goal is to classify ordinary theorems in terms of their provability strength.
It uses the framework of subsystems of second order arithmetic,
which is sufficiently rich to express many theorems in a natural way.
The base system, $\rca$ standing for Recursive Comprehension Axiom,
contains the basic first order Peano arithmetic together with the~$\Delta^0_1$
comprehension scheme and the~$\Sigma^0_1$ induction scheme.
Thanks to the equivalence between~$\Delta^0_1$-definable sets
and computable sets, $\rca$ can be considered as capturing
``computable mathematics''. The proof-theoretic
analysis of the theorems in reverse mathematics is therefore
closely related to their computational analysis. See Simpson~\cite{Simpson2009Subsystems}
for a formal introduction to reverse mathematics.

Early reverse mathematics have led to two main empirical observations:
First, many ordinary (i.e.\ non set-theoretic) theorems require very weak set existence axioms.
Second, most of those theorems are in fact \emph{equivalent} to one of five
main subsystems, known as the ``Big Five''.
However, among the theorems studied in reverse mathematics, a notable class of theorems
fails to support the second observation, namely, Ramsey-type theorems.
The underlying idea of Ramsey's theory is that whenever a collection of
objects is sufficiently large, we can always find an arbitrarily large
sub-collection of objects satisfying some given structural property.
Perhaps the most well-known statement is Ramsey's theorem,
stating that every coloring of tuples of integers with a finite 
number of colors admits an infinite monochromatic subset.
The various consequences of Ramsey's theorem usually fail to coincide with
the main five subsystems, and slight variations of their statements lead to
different subsystems. The study of Ramsey-type statements
has been a very active research subject in reverse mathematics over the past few years~\cite{Bienvenu2015logical,Cholak2001strength,FriedmanFom53free,Hirschfeldt2007Combinatorial}.
See Hirschfeldt~\cite{Hirschfeldt2015Slicing} for a good introduction to recent reverse mathematics.

\subsection{Reducibilities}\label{subsect:intro-reducibilities}

Many theorems in reverse mathematics are~$\Pi^1_2$ statements, i.e., of the form $(\forall X)(\exists Y)\Phi(X, Y)$
where $\Phi$ is an arithmetic formula.
They can be considered as \emph{problems} which usually come with a natural
class of~\emph{instances}. Given an instance~$X$, a set~$Y$ such that~$\Phi(X,Y)$ holds
is called a~\emph{solution} to~$X$. For example, König's lemma states
that every infinite, finitely branching tree has an infinite path. In this statement,
an instance is a infinite, finitely branching tree~$T$, and a solution to~$T$
is an infinite path through~$T$.

Thanks to the computational nature of the axioms of~$\rca$,
given two~$\Pi^1_2$ statements~$\Psf$ and~$\Qsf$, a proof of implication
~$\Qsf \imp \Psf$ consists in taking an arbitrary~$\Psf$-instance~$I$
and computing a solution to~$I$ in a computational process involving
several applications of the~$\Qsf$ principle. If the proof relativizes
and can be formalized over~$\rca$ (the main concern being the restriction to~$\Sigma^0_1$-induction),
we obtain a proof of~$\rca \vdash \Qsf \imp \Psf$.
It is often the case that the proof of the implication $\Qsf \imp \Psf$ involves
only one application of~$\Qsf$ given an instance of~$\Psf$. Such a reduction
is called a~\emph{computable reduction}.

\begin{definition}[Computable reducibility] Fix two~$\Pi^1_2$ statements~$\Psf$ and $\Qsf$.
\begin{itemize}
	\item[1.] $\Psf$ is \emph{computably reducible} to~$\Qsf$ (written $\Psf \leq_c \Qsf$)
if every~$\Psf$-instance~$I$ computes a~$\Qsf$-instance~$J$ such that for every solution~$X$ to~$J$,
$X \oplus I$ computes a solution to~$I$.
	\item[2.] $\Psf$ is \emph{strongly computably reducible} to a~$\Qsf$ (written~$\Psf \leq_{sc} \Qsf$)
if every~$\Psf$-instance~$I$ computes a~$\Qsf$-instance~$J$ such that every solution to~$J$
computes a solution to~$I$.
\end{itemize}
\end{definition}

Of course, proving that a statement~$\Psf$ is not computably reducible to another statement~$\Qsf$
is not sufficient for separating the statements over~$\rca$.
For example, we shall see that Ramsey's theorem for pairs with $k+1$ colors is not computably
reducible to Ramsey's theorem for pairs with $k$ colors, whereas the statements are known
to be logically equivalent over~$\rca$. However, proving that~$\Psf \not \leq_c \Qsf$ can be seen
as a preliminary step towards the separation of the principles. Lerman et al.~\cite{Lerman2013Separating}
have developped a framework for iterating a one-step non-reducibility
into a separation over~$\rca$.

Other reducibility notions have been introduced to better understand the computational
content of theorems from the point of view of reverse mathematics. Dorais et al.~\cite{Dorais2016uniform}
studied the uniformity of the computable reductions $\Psf \leq_c \Qsf$ by requiring
the construction of a~$\Qsf$-instance~$J$ given a~$\Psf$-instance~$I$ and the construction of a solution to~$I$
given a solution to~$J$ to be done with two fixed Turing functionals. They showed that this \emph{uniform reducibility}
is the restriction of the Weihrauch reduction to the second-order setting. 
Hirschfeldt and Jockusch~\cite{Hirschfeldtnotions} introduced a game-theoretic approach
and defined a \emph{generalized uniform reducibility} extending the notion of uniform reducibility
to several applications of the statement~$\Qsf$.
In this paper, we shall restrict ourselves to computable reducibility and provability over~$\rca$.

\subsection{Degrees of unsolvability of cohesiveness}

Cohesiveness plays a central role in reverse mathematics.
It appears naturally in the standard proof of Ramsey's theorem, as a preliminary
step to reduce an instance of Ramsey's theorem over $(n+1)$-tuples into a non-effective instance over $n$-tuples.
An important part of current research about Ramsey-type principles
in reverse mathematics consists in trying to understand whether
cohesiveness is a consequence of stable Ramsey's theorem for pairs,
or more generally whether it is a combinatorial consequence of the infinite pigeonhole principle~\cite{Cholak2001strength,Dzhafarov2014Cohesive,Dzhafarov2012Cohesive,Wang2013Omitting}.
Chong et al.~\cite{Chong2014metamathematics} recently showed using non-standard models that cohesiveness
is not a proof-theoretic consequence of the pigeonhole principle.
However it is not known whether or not cohesiveness is computably reducible to stable Ramsey's theorem for pairs.

\begin{definition}[Cohesiveness]
An infinite set $C$ is $\vec{R}$-cohesive for a sequence of sets $R_0, R_1, \dots$
if for each $i \in \omega$, $C \subseteq^{*} R_i$ or $C \subseteq^{*} \overline{R_i}$.
A set $C$ is \emph{p-cohesive} if it is $\vec{R}$-cohesive where
$\vec{R}$ is an enumeration of all primitive recursive sets.
$\coh$ is the statement ``Every uniform sequence of sets $\vec{R}$
has an $\vec{R}$-cohesive set.''
\end{definition}

Jockusch and Stephan~\cite{Jockusch1993cohesive} studied the degrees of unsolvability of cohesiveness
and proved that~$\coh$ admits a universal instance whose solutions
are the p-cohesive sets. They characterized their degrees as those whose
jump is PA relative to~$\emptyset'$.

Cohesiveness is a $\Pi^1_2$ statement whose instances are sequences of sets~$\vec{R}$
and whose solutions are~$\vec{R}$-cohesive sets. It is natural to wonder about the degrees
of unsolvability of the~$\vec{R}$-cohesive sets according to the sequence of sets~$\vec{R}$.
Mingzhong Cai asked whether whenever a uniformly computable sequence of sets~$R_0$, $R_1, \dots$
has no computable~$\vec{R}$-cohesive set, there exists a non-computable set which does not compute
one. In the opposite direction, one may wonder whether every unsolvable instance of~$\coh$ is maximally difficult.
A natural first approach in the analysis of the strength of a principle consists
in looking in which way typical sets can help in computing
a solution to an unsolvable instance.
The notion of typical set is usually understood in two different ways: using the genericity approach
and the randomness approach. Wang~\cite{Wang2013Omitting} answered Cai's question
 by investigating the solvability of cohesiveness by typical sets.

In this paper, we refine Wang's analysis by establishing a pointwise correspondence
between sets cohesive for a sequence and sets whose jump computes a member of a~$\Pi^{0, \emptyset'}_1$ class.
Then, using the known interrelations between typical sets and~$\Pi^0_1$ classes,
we give precise genericity and randomness bounds above which no typical set helps computing a cohesive set.
We identify different layers of unsolvability and spot a class of instances sharing many
properties with the universal instance.
Emulating work in~\cite{PateyCombinatorial} on the pigeonhole principle and
weak K\"onig's lemma ($\wkl$), 
we show that some unsolvable instances of~$\coh$ are combinatorial consequences of the pigeonhole principle.

\subsection{Ramsey's theorem and computable reducibility}

The strength of Ramsey-type statements is notoriously hard to tackle
in the setting of reverse mathematics. The separation of Ramsey's theorem for pairs ($\rt^2_2$)
from the arithmetical comprehension axiom ($\aca$) was a long-standing open
problem, until Seetapun and Slaman solved it~\cite{Seetapun1995strength} with his notion of cone avoidance.
The question of the relation between $\rt^2_2$ and weak König's lemma ($\wkl$)
remained open for many years before Cholak, Jockusch and Slaman~\cite{Cholak2001strength} proved that~$\wkl$
does not imply~$\rt^2_2$ over~$\rca$. More than fifteen years after Seetapun, Liu~\cite{Liu2012RT22} solved the remaining direction
by proving that~$\rt^2_2$ does not imply~$\wkl$ over~$\rca$.

\begin{definition}[Ramsey's theorem]
A subset~$H$ of~$\omega$ is~\emph{homogeneous} for a coloring~$f : [\omega]^n \to k$ (or \emph{$f$-homogeneous}) 
if each $n$-tuple over~$H$ is given the same color by~$f$. A coloring $f : [\omega]^{n+1} \to k$ is \emph{stable} if for every~$n$-tuple~$\sigma \in [\omega]^n$,
$\lim_s f(\sigma, s)$ exists.
$\rt^n_k$ is the statement ``Every coloring $f : [\omega]^n \to k$ has an infinite $f$-homogeneous set''.
$\srt^n_k$ is the restriction of~$\rt^n_k$ to stable colorings.
\end{definition}

Simpson~\cite[Theorem III.7.6]{Simpson2009Subsystems} proved that whenever~$n \geq 3$ and~$k \geq 2$,
$\rca \vdash \rt^n_k \biimp \aca$. Ramsey's theorem for pairs is probably the most famous 
example of statement escaping the Big Five. Seetapun~\cite{Seetapun1995strength} proved that~$\rt^2_2$ 
is strictly weaker than~$\aca$ over~$\rca$. Because of the complexity of the related separations, 
$\rt^2_2$ received a particular attention from the reverse mathematics community. 
Mileti~\cite{Mileti2004Partition} and Jockusch and Lempp [unpublished]
proved that~$\rt^2_2$ is equivalent to~$\srt^2_2+\coh$ over~$\rca$. Recently,
Chong et al.~\cite{Chong2014metamathematics} proved that~$\srt^2_2$ is strictly weaker than~$\rt^2_2$
over~$\rca$. However they used non-standard models to separate the statements and the question
whether~$\srt^2_2$ and~$\rt^2_2$ coincide over~$\omega$-models remains open.
Stable Ramsey's theorem for pairs can be characterized by a purely computability-theoretic statement.

\begin{definition}
For every~$n, k \geq 1$, 
$\mathsf{D}^n_k$ is the statement ``Every~$\Delta^0_n$ $k$-partition of the integers has an infinite subset of one of its parts''.
$\mathsf{D}^n_{<\infty}$ is the statement~$(\forall k)\mathsf{D}^n_k$.
\end{definition}

Cholak et al.~\cite{Cholak2001strength} proved that~$\mathsf{D}^2_k$ and~$\srt^2_k$ are computably equivalent
and that the proof is formalizable over~$\rca+\bst$.
Later, Chong et al.~\cite{Chong2010role} proved that~$\mathsf{D}^2_2$ implies~$\bst$ over~$\rca$,
showing therefore that~$\rca \vdash \mathsf{D}^2_k \biimp \srt^2_\ell$ for every~$k, \ell \geq 2$.
Dzhafarov~\cite{Dzhafarov2014Cohesive} proved that~$\coh \not \leq_{sc} \mathsf{D}^2_{<\infty}$ by constructing 
a sequence of sets~$R_0, R_1, \dots$ such that for every~$k \geq 2$,
every instance of~$\rt^1_k$ hyperarithmetic in~$\vec{R}$ has a solution which does not compute an~$\vec{R}$-cohesive set.
In section~\ref{sect:coh-strong-reduc}, we strengthen this result by making~$\vec{R}$ uniformly $\Delta^0_2$
and removing the effectiveness restriction on the instance of~$\rt^1_k$. The proof reveals the combinatorial nature
of the relations between cohesiveness and~$\rt^1_k$ and answers a question of Antonio Mont\'alban.
Recently, Dzhafarov~\cite{DzhafarovStrong} proved that~$\coh \not \leq_{sc} \srt^2_2$.

Another closely related subject of interest is the impact of the number of colors in the strength of Ramsey's theorem.
For every~$n \geq 1$ and~$k, \ell \geq 2$, $\rca \vdash \rt^n_k \biimp \rt^n_\ell$ by a simple color blindness argument.
Whenever $k > \ell \geq 2$, the reduction of~$\rt^n_k$ to~$\rt^n_\ell$ involves more than one application of~$\rt^n_\ell$
and therefore is not a computable reduction.
Hirschfeldt and Jockusch~\cite{Hirschfeldtnotions} noticed that the proof of Dzhafarov~\cite{Dzhafarov2014Cohesive} can be modified
to obtain~$\rt^1_k \not \leq_{sc} \rt^1_\ell$. 
Dorais et al.~\cite{Dorais2016uniform} asked in which case~$\rt^n_k \leq_c \rt^n_\ell$.
In section~\ref{sect:rt-reducibility}, we answer 
by proving that $\srt^n_k \not \leq_c \rt^n_\ell$ whenever~$k > \ell \geq 2$ and~$n \geq 2$.

\subsection{The weakness of free set and thin set theorems}

Simpson~\cite[Theorem III.7.6]{Simpson2009Subsystems} proved
that the hiearchy of Ramsey's theorem collapses at level three
in reverse mathematics.
One may wonder about some natural weakenings of Ramsey's theorem over arbitrary tuples
which remain strictly weaker than~$\aca$. 
Given a coloring~$f : [\omega]^n \to k$, instead of stating the existence of an infinite $f$-homogeneous set~$H$,
we can simply require that $f$ \emph{avoids at least one color} over the set~$H$. This is the notion of $f$-thin set.

\begin{definition}[Thin set theorem]
Given a coloring~$f : [\omega]^n \to k$ (resp.\ $f : [\omega]^n \to \omega$), an infinite set~$H$
is \emph{thin} for~$f$ if~$|f([H]^n)| \leq k-1$ (resp.\ $f([H]^n) \neq \omega$).
For every~$n \geq 1$ and~$k \geq 2$, $\ts^n_k$ is the statement
``Every coloring $f : [\omega]^n \to k$ has a thin set''
and $\ts^n$ is the statement ``Every coloring $f : [\omega]^n \to \omega$ has a thin set''.
$\sts^n_k$ is the restriction of~$\ts^n_k$ to stable colorings.
$\ts$ is the statement~$(\forall n)\ts^n$.
\end{definition}

The reverse mathematical analysis of the thin set theorem started with Friedman~\cite{FriedmanFom53free,Friedman2013Boolean}.
It has been studied by Cholak et al.~\cite{Cholak2001Free},
Wang~\cite{Wang2014Some} and the author~\cite{PateyCombinatorial,Patey2015Degrees} among others.
Dorais et al.~\cite{Dorais2016uniform} proved that~$\ts^1_k$ is not uniformly reducible
to~$\ts^1_\ell$ whenever~$\ell > k$. Hirschfeldt and Jockusch~\cite{Hirschfeldtnotions}
extended the result to colorings over arbitrary tuples.
We generalize the previous theorems by proving that $\ts^n_k \not \leq_c \ts^n_\ell$
whenever~$\ell > k \geq 2$ and~$n \geq 2$. In the case of colorings of singletons,
we prove that~$\ts^1_k \not \leq_{sc} \ts^1_\ell$ whenever~$\ell > k \geq 2$.

The free set theorem is a strengthening of the thin set theorem
in which every member of a free set is a witness of thinness of the same set.
Indeed, if~$H$ is an infinite~$f$-free set for some function~$f$,
for every~$a \in H$, $H \setminus \{a\}$ is $f$-thin with witness color~$a$.
See Theorem~3.2 in~\cite{Cholak2001Free} for a formal version of this claim.

\begin{definition}[Free set theorem]
Given a coloring $f : [\omega]^n \to \omega$, an infinite set~$H$
is \emph{free} for~$f$ if for every~$\sigma \in [H]^n$,
$f(\sigma) \in H \imp f(\sigma) \in \sigma$.
For every~$n \geq 1$, $\fs^n$ is the statement
``Every coloring $f : [\omega]^n \to \omega$ has a free set''.
$\sfs^n$ is the restriction of~$\fs^n$ to stable colorings.
$\fs$ is the statement~$(\forall n)\fs^n$.
\end{definition}

Cholak et al.~\cite{Cholak2001Free} proved that~$\rca \vdash \rt^n_2 \imp \fs^n \imp \ts^n$ for every~$n \geq 2$.
Wang~\cite{Wang2014Some} proved that~$\fs$ (hence~$\ts$) does not imply~$\aca$
over~$\omega$-models. The author~\cite{PateyCombinatorial} proved
that~$\fs$ does not imply~$\wkl$ (and in fact weak weak König's lemma) over~$\rca$.

Cholak et al.~\cite{Cholak2001Free} and Montalban~\cite{Montalban2011Open} asked whether any of~$\ts^2$, $\fs^2$, $\fs^2+\coh$ and
$\fs^2+\wkl$ imply~$\rt^2_2$ over~$\rca$. Hirschfeldt~\cite{Hirschfeldt2015Slicing} asked whether~$\fs^2+\wkl$
implies any of~$\srt^2_2$, the ascending descending sequence ($\ads$) and the chain antichain principle~($\cac$).
We answer all these questions negatively by proving that for every~$k \geq 2$,
the conjunction of $\coh$, $\wkl$,
the Erd\H{o}s-Moser theorem ($\emo$), $\ts^2_{k+1}$, $\fs$ and~$\ts$ implies neither~$\sts^2_k$
nor the stable ascending descending sequence ($\sads$) over~$\rca$.

\subsection{Organization of the paper}

In section~\ref{sect:coh-degrees-unsolvability}, we establish an instance-wise correspondence
between cohesive sets and degrees whose jump computes a member of a $\Pi^{0,\emptyset'}_1$ class.
We take advantage of this correspondence to study how typical sets are useful to compute unsolvable instances of cohesiveness,
and extend this analysis to Ramsey-type statements.
In section~\ref{sect:coh-strong-reduc}, we reprove Dzhafarov's result that cohesiveness is not strongly computably
reducible to $\mathsf{D}^2_{<\infty}$ with a more combinatorial proof using hyperimmunity.
In section~\ref{sect:rt-reducibility}, we refine the forcing of the previous section
to separate Ramsey's theorem over computable reducibility according to the number of colors.
In section~\ref{sect:ts-strong-reduc}, we separate variants of the thin set theorem for singletons over strong computable 
reducibility according to the number of colors using preservation of non-c.e.\ definitions.
Finally, we separate the thin set theorem for pairs from Ramsey's theorem for pairs
over~$\rca$ in section~\ref{sect:preservation-ts2-omega-models}, and extend this separation
to the full thin set theorem in section~\ref{sect:preservation-ts-omega-models} 
and the full free set theorem in section~\ref{sect:fs-omega-models}.

\subsection{Notation}

\emph{String, sequence}.
Fix an integer $k \in \omega$.
A \emph{string} (over $k$) is an ordered tuple of integers $a_0, \dots, a_{n-1}$
(such that $a_i < k$ for every $i < n$). The empty string is written~$\varepsilon$. A \emph{sequence}  (over $k$)
is an infinite listing of integers $a_0, a_1, \dots$ (such that $a_i < k$ for every $i \in \omega$).
Given $s \in \omega$,
$k^s$ is the set of strings of length $s$ over~$k$ and
$k^{<s}$ is the set of strings of length $<s$ over~$k$. Similarly,
$k^{<\omega}$ is the set of finite strings over~$k$
and $k^{\omega}$ is the set of sequences (i.e. infinite strings)
over~$k$. 
If $\sigma$ is a string, then $|\sigma|$ denotes its length.
Given two strings $\sigma, \tau \in k^{<\omega}$, $\sigma$ is a \emph{prefix}
of $\tau$ (written $\sigma \preceq \tau$) if there exists a string $\rho \in k^{<\omega}$
such that $\sigma \rho = \tau$. Given a sequence $X$, we write $\sigma \prec X$ if
$\sigma = X \uh n$ for some $n \in \omega$, where $X \uh n$ denotes the restriction of $X$ to its first $n$ elements.
A \emph{binary string} (resp. real) is a \emph{string} (resp. sequence) over $2$.
We may identify a real with a set of integers by considering that the real is its characteristic function.
Accordingly, we identify a string~$\sigma \in 2^{<\omega}$ 
with the set~$set(\sigma) = \{n  < |\sigma| : \sigma(n) = 1\}$.
Therefore~$n \in \sigma$ means~$n \in set(\sigma)$ and
given a set~$A \subseteq \omega$, we denote by~$\sigma \cap A$ the string~$\tau \in 2^{|\sigma|}$
such that~$\tau(n) = 1$ if and only if~$\sigma(n) = 1$ and $n \in A$.
We also write~$\sigma \subseteq A$ for~$set(\sigma) \subseteq A$.
Given a real~$X \in 2^\omega$ and a string~$\sigma$, we denote by~$X/\sigma$
the real obtained by replacing the~$|\sigma|$ first bits of~$X$ by~$\sigma$.

\emph{Tree, path}.
A tree $T \subseteq \omega^{<\omega}$ is a set downward closed under the prefix relation.
The tree~$T$ is \emph{finitely branching} if every node~$\sigma \in T$
has finitely many immediate successors.
A \emph{binary} tree is a tree~$T \subseteq 2^{<\omega}$.
A set $P \subseteq \omega$ is a \emph{path} through~$T$ if for every $\sigma \prec P$,
$\sigma \in T$. A string $\sigma \in k^{<\omega}$ is a \emph{stem} of a tree $T$
if every $\tau \in T$ is comparable with~$\sigma$.
Given a tree $T$ and a string $\sigma \in T$,
we denote by $T^{[\sigma]}$ the subtree $\{\tau \in T : \tau \preceq \sigma \vee \tau \succeq \sigma\}$.
We write~$P \gg X$ to say that~$P$ is of PA degree relative to~$X$.

\emph{Classes}.
Given a finite string~$\sigma \in \omega^{<\omega}$, $[\sigma]$ is the set of sequences extending~$\sigma$.
Whenever it is clear from the context that we are working with binary strings, $[\sigma]$
denotes the set of \emph{reals} extending~$\sigma$.
A~\emph{$\Pi^{0,X}_1$ class} is the collection of paths through an $X$-computable tree. The complement of a~$\Pi^{0,X}_1$ class
is a~\emph{$\Sigma^{0,X}_1$ class}. A~\emph{$k$-enum} of a class~$\Ccal \subseteq 2^\omega$
is a uniform sequence of finite sets of strings~$D_0, D_1, \dots$ such that~$D_s$
is a set of~at most $k$ binary strings of length~$s$ such that one of those is a prefix of a member of~$\Ccal$.

\emph{Sets, partitions}.
Given two sets $A$ and $B$, we denote by $A < B$ the formula
$(\forall x \in A)(\forall y \in B)[x < y]$
and by $A \subseteq^{*} B$ the formula $(\exists b)(\forall x \in A)[x \not \in B \imp x < b]$,
meaning that $A$ is contained in $B$ except for at most finitely many elements.
Given a set~$X$ and some integer~$k$, a~\emph{$k$-cover of~$X$}
is a $k$-uple $A_0, \dots, A_{k-1}$ such that~$A_0 \cup \dots \cup A_{k-1} = X$.
We may simply say~\emph{$k$-cover} when the set~$X$ is unambiguous. 
A \emph{$k$-partition} is a $k$-cover whose sets are pairwise disjoint.

\section{The degrees of unsolvability of cohesiveness}\label{sect:coh-degrees-unsolvability}

In this section, we study the degree of unsolvability of $\vec{R}$-cohesive sets
according the degree of unsolvability of the sequence $\vec{R}$ itself.
Then we take advantage of this
analysis to answer various questions about which theorems in reverse mathematics
can solve a computably unsolvable instance of cohesiveness.

\subsection{Cohesiveness and \texorpdfstring{$\Pi^{0,\emptyset'}_1$}{Pi02} classes}\label{sect:coh-delta2-trees}

Jockusch and Stephan characterized the p-cohesive degrees
as those whose jump is of degree PA relative to~$\emptyset'$.
We clarify the situation by establishing 
an instance-wise correspondence between the degrees
of the sets cohesive for a sequence, and the degrees whose jump
computes a member of a non-empty $\Pi^{0, \emptyset'}_1$ class.

\begin{definition}
Let $R_0, R_1, \dots$ be a uniformly computable sequence of sets.
For every~$\sigma \in 2^{<\omega}$, we define $R_\sigma$ inductively as follows.
First, $R_\varepsilon = \omega$ and then, if $R_\sigma$ has already been defined for some string~$\sigma$ of length~$s$,
let $R_{\sigma 0} = R_\sigma \cap \overline{R_s}$ and~$R_{\sigma 1} = R_\sigma \cap R_s$.
For example, $R_{0110} = \overline{R_0} \cap R_1 \cap R_2 \cap \overline{R_3}$.
Let $\Ccal(\vec{R})$ be the $\Pi^{0,\emptyset'}_1$ class of binary sequences~$P$
such that for every~$\sigma \prec P$, the set~$R_\sigma$ is infinite.
\end{definition}

Our first lemma shows that the degrees of~$\vec{R}$-cohesive
sets can be characterized by their jumps.
This lemma reveals in particular that low sets fail to solve unsolvable instances
of cohesiveness.

\begin{lemma}\label{lem:coh-tree}
For every uniformly computable sequence of sets~$R_0, R_1, \dots$,
a set computes an~$\vec{R}$-cohesive set if and only if
its jump computes a member of~$\Ccal(\vec{R})$.
\end{lemma}
\begin{proof}
Fix an~$\vec{R}$-cohesive set~$C$.
Let~$P = \bigcup \{ \sigma \in 2^{<\omega} : C \subseteq^{*} R_\sigma \}$.
The sequence~$P$ is infinite and~$C'$-computable as there exists exactly
one string~$\sigma$ of each length such that~$C \subseteq^{*} R_\sigma$.
In particular, for every~$\sigma \prec P$, $R_\sigma$ is infinite,
so~$P$ is a member of~$\Ccal(\vec{R})$.

Conversely, let~$X$ be a set whose jump computes a member~$P$ of~$\Ccal(\vec{R})$.
By Schoenfield's limit lemma~\cite{Shoenfield1959degrees}, there exists an~$X$-computable function~$f(\cdot, \cdot)$
such that for each~$x$, $\lim_s f(x, s) = P(x)$. Define an~$\vec{R}$-cohesive set~$C = \bigcup_s C_s$
$X$-computably by stages $C_0 = \emptyset \subsetneq C_1 \subsetneq \dots$ as follows. 
At stage~$s$, search for some string~$\sigma$ of length~$s$
and some integer~$n \in R_\sigma$ greater than~$s$ such that~$f(x, n) = \sigma(x)$ for each~$x < |\sigma|$.
We claim that such~$\sigma$ and~$n$ must exist, as there exists a threshold~$n_0$
such that for every~$n > n_0$, $f(x, n) = P(x)$ for each~$x < s$.
Let~$\sigma \prec P$ be of length~$s$. By definition of~$P$, $R_\sigma$ is infinite,
so there must exist some $n \in R_\sigma$ which is greater than~$n_0$ and~$s$.
Set~$C_{s+1} = C_s \cup \{n\}$ and go to the next stage.
We now check that~$C = \bigcup_s C_s$ is~$\vec{R}$-cohesive.
For every~$x \in \omega$, there exists a threshold~$n_1$ such that
for every~$n > n_1$, $f(x, n) = P(x)$. By construction, for every element~$n \in C \setminus C_{n_1}$,
$n \in R_\sigma$ for some string~$\sigma$ such that~$\sigma(x) = P(x)$. Therefore~$C \subseteq^{*} R_x$
or~$C \subseteq^{*} \overline{R_x}$.
\end{proof}

Jockusch and Stephan~\cite{Jockusch1993cohesive}
showed the existence of a uniformly computable sequence of sets~$R_0, R_1, \dots$
having no low~$\vec{R}$-cohesive set. We prove that it suffices
to consider any sequence~$\vec{R}$ with no computable~$\vec{R}$-cohesive set to obtain this property.

\begin{corollary}\label{cor:low-helps-not-coh}
A uniformly computable sequence of sets~$R_0, R_1, \dots$ has
a low~$\vec{R}$-cohesive set if and only if it has a computable~$\vec{R}$-cohesive set.
\end{corollary}
\begin{proof}
Let~$X$ be a low $\vec{R}$-cohesive set. By Lemma~\ref{lem:coh-tree},
the jump of~$X$ (hence~$\emptyset'$) computes a member of~$\Ccal(\vec{R})$.
By a second application of Lemma~\ref{lem:coh-tree}, 
the existence of a computable~$\vec{R}$-cohesive set follows.
\end{proof}

One may naturally wonder about the shape of the~$\Pi^{0,\emptyset'}_1$
classes~$\Ccal(\vec{R})$ for uniformly computable sequences $R_0, R_1, \dots$
We show through the following lemma that~$\Ccal(\vec{R})$ can be any~$\Pi^{0,\emptyset'}_1$
class. Together with Lemma~\ref{lem:coh-tree}, it establishes an instance-wise correspondence
between cohesive sets and~$\Pi^{0,\emptyset'}_1$ classes.

\begin{lemma}\label{lem:tree-coh-inv}
For every non-empty~$\Pi^{0,\emptyset'}_1$ class~$\Dcal \subseteq 2^{\omega}$,
there exists a uniformly computable sequence of sets~$R_0, R_1, \dots$
such that~$\Ccal(\vec{R}) = \Dcal$.
\end{lemma}
\begin{proof}
By Schoenfield's limit lemma~\cite{Shoenfield1959degrees}, 
there exists a computable function~$g : 2^{<\omega} \times \omega \to 2$
whose limit exists and such that~$\Dcal$ is the collection of~$X$ such that 
for every~$\sigma \prec X$, $\lim_s g(\sigma, s) = 1$. We can furthermore assume
that whenever~$g(\sigma, s) = 1$, then for every~$\tau \prec \sigma$, $g(\tau, s) = 1$,
and that for every~$s \in \omega$, the set~$U_s = \{ \sigma \in 2^s : g(\sigma, s) = 1\}$ is non-empty.
We define a uniformly computable sequence of sets~$R_0, R_1, \dots$ such that~$\Ccal(\vec{R}) = \Dcal$
by stages as follows. 

As stage~$0$, $R_i = \emptyset$ for every~$i \in \omega$.
Suppose that we have already decided~$R_i \restr n_s$ for every~$i \in \omega$ and some~$n_s \in \omega$.
At stage~$s+1$, we will add elements to~$R_0, \dots, R_s$ so that for each string~$\sigma$ of length~$s+1$,
$R_\sigma \restr [n_s, n_s+p] \neq \emptyset$ if and only if $\sigma \in U_{s+1}$.
To do so, consider the set~$U_{s+1} = \{\sigma_0, \dots, \sigma_p\}$ defined above
and add $\{n_s + i : \sigma_i(j) = 1, i \leq p \}$ to $R_j$ for each~$j \leq s$.
Set~$n_{s+1} = n_s + p + 1$ and go to the next stage.

We claim that $R_\sigma$ is infinite if and only if~$\sigma \prec X$ for some~$X \in \Dcal$.
Assume that~$R_\sigma$ is infinite. By construction, there are infinitely many~$s$
such that~$R_\sigma \restr [n_s, n_s+p] \neq \emptyset$. So there are infinitely many stages~$s$ such that
$\tau \in U_s$ ($g(\tau, s) = 1$) for some~$\tau \succeq \sigma$. By assumption on~$g$,
there are infinitely many~$\tau \succeq \sigma$ such that $g(\tau, s) = 1$ for infinitely many~$s$.
Therefore, by compactness, there exists some~$X \in \Dcal$ such that~$\sigma \prec X$.
Conversely, if~$\sigma \prec X$ for some~$X \in \Dcal$, then there are infinitely many stages~$s$
such that~$\tau \in U_s$ for some~$\tau \succeq \sigma$. At each of these stages,
$R_\sigma \restr [n_s, n_s + p] \supseteq R_\tau \restr [n_s, n_s + p] \neq \emptyset$.
Therefore~$R_\sigma$ is infinite.
\end{proof}

Jockusch et al.\ proved in~\cite{Jockusch1991} that for every~$\Pi^{0,\emptyset'}_1$ class $\Ccal \subseteq 2^{\omega}$,
there exists a~$\Pi^0_1$ class~$\Dcal \subseteq \omega^\omega$ such that~$\deg(\Ccal) = \deg(\Dcal)$,
where~$\deg(\Ccal)$ is the class of degrees of members of~$\Ccal$.
For the reader who is familiar with Weihrauch degrees, what we actually prove here is that
König's lemma is the jump of the cohesiveness principle under Weihrauch reducibility.
Bienvenu [personal communication] suggested the use of Simpson's Embedding Lemma~\cite[Lemma 3.3]{Simpson2007extension}
to prove the reducibility of some unsolvable instances of cohesiveness to various statements.

\begin{lemma}[Bienvenu]\label{lem:sigma3-coh}
For every~$\Sigma^{0,\emptyset'}_3$ class~$\Ecal \subseteq \omega^\omega$ with no $\emptyset'$-computable member,
there exists a uniformly computable sequence of sets~$R_0, R_1, \dots$
with no computable~$\vec{R}$-cohesive set but such that every member of~$\Ecal$ computes an~$\vec{R}$-cohesive set.
\end{lemma}
\begin{proof}
By a relativization of~Lemma~3.3 in~\cite{Simpson2007extension}, there exists a~$\Pi^{0, \emptyset'}_1$ class~$\Dcal$
whose degrees (relative to~$\emptyset'$) are exactly $\deg(\Ecal) \cup PA[\emptyset']$. Therefore~$\Dcal$ has no~$\emptyset'$-computable member and every member of~$\Ecal$ $\emptyset'$-computes a member of~$\Dcal$.
By Lemma~\ref{lem:tree-coh-inv}, there exists a uniformly computable sequence of sets~$R_0, R_1, \dots$
such that~$\Ccal(\vec{R}) = \Dcal$. By Lemma~\ref{lem:coh-tree}, there exists no computable $\vec{R}$-cohesive set,
but every member of~$\Dcal$ (and in particular every member of~$\Ecal$) computes an~$\vec{R}$-cohesive set.
\end{proof}

\subsection{How genericity helps solving cohesiveness}\label{sect:coh-genericity} 

We now take advantage of the analysis of the previous section to deduce optimal bounds
on how much genericity is needed to avoid solving an unsolvable instance of~$\coh$.

\begin{definition}[Genericity]
Fix a set of strings~$S \subseteq 2^{<\omega}$.
The set~$S$ is \emph{dense} if every string has an extension in~$S$.
A real~$G$ \emph{meets}~$S$ if it has some initial segment in~$S$.
A real~$G$ \emph{avoids}~$S$ if it has an initial segment with no extension in~$S$.
Given an integer~$n \in \omega$, a real is~\emph{$n$-generic} if it meets or avoids each~$\Sigma^0_n$ set of strings.
A real is~\emph{weakly $n$-generic} if it meets each~$\Sigma^0_n$ dense set of strings.
\end{definition}

By Friedberg's jump inversion theorem~\cite{Friedberg1957criterion}, there exists a 1-generic which is of high degree,
and therefore computes a cohesive set for every uniformly computable sequence of sets.
Wang~\cite{Wang2013Omitting} proved that whenever a uniformly computable sequence of sets~$R_0, R_1, \dots$
has no computable~$\vec{R}$-cohesive sets, no weakly 3-generic computes an~$\vec{R}$-cohesive set.
He asked whether there exists a 2-generic computing an~$\vec{R}$-cohesive set.
We prove the optimality of Wang's bound by showing the existence of an unsolvable instance
of~$\coh$ which is solvable by a 2-generic real.

\begin{lemma}
There exists a 2-generic real~$G$ together with a
uniformly computable sequence of sets~$R_0, R_1, \dots$
with no computable~$\vec{R}$-cohesive set such that
$G$ computes an~$\vec{R}$-cohesive set.
\end{lemma}
\begin{proof}
Fix any~$\Delta^0_3$ 2-generic real~$G$ and consider the singleton~$\Ecal = \{G\}$.
As no 2-generic is~$\Delta^0_2$, the class $\Ecal$ has no~$\emptyset'$-computable member.
By Lemma~\ref{lem:sigma3-coh}, there exists a uniformly computable sequence of sets
$R_0, R_1, \dots$ with no computable~$\vec{R}$-cohesive set, 
such that~$G$ computes an~$\vec{R}$-cohesive set.
\end{proof}

However, if we slightly increase the unsolvability of the sequence of sets,
no 2-generic real helps computing a set cohesive for the sequence.
Recall that a~\emph{1-enum} of a class~$\Ccal \subseteq 2^{<\omega}$
is a sequence of strings~$\sigma_0, \sigma_1, \dots$
such that~$|\sigma_s| = s$ and~$[\sigma_s] \cap \Ccal \neq \emptyset$ for each~$s \in \omega$.
The notion has been extensively studied in~\cite{PateyCombinatorial}.

\begin{theorem}
For any uniformly computable sequence of sets~$R_0, R_1, \dots$
such that~$\Ccal(\vec{R})$ has no~$\emptyset'$-computable 1-enum,
no 2-generic real computes an~$\vec{R}$-cohesive set.
\end{theorem}
\begin{proof}
By Jockusch~\cite{JockuschJr1980Degrees}, every $n$-generic set is~GL${}_n$
and in particular, every 2-generic is~GL${}_1$. Therefore,
by Lemma~\ref{lem:coh-tree},
a 2-generic set~$G$ computes an~$\vec{R}$-cohesive set
if and only if there exists some functional~$\Gamma$
such that~$\Gamma^{G \oplus \emptyset'}$ is a member of~$\Ccal(\vec{R})$.
Fix a functional~$\Gamma$ such that~$\Gamma^{G \oplus \emptyset'}$ is total for some
2-generic set~$G$, and define the following $\Sigma^{0, \emptyset'}_1$ set:
\[
W_{bad} = \{ \sigma \in 2^{<\omega} : [\Gamma^{\sigma \oplus \emptyset'}] \cap \Ccal(\vec{R}) = \emptyset \}
\]

We claim $G$ meets~$W_{bad}$.
Suppose for contradiction that $G$ avoids $W_{bad}$. By 2-genericity of~$G$, there exists a string~$\sigma \prec G$
with no extension in~$W_{bad}$. We show that there exists a~$\emptyset'$-effective
procedure which computes a 1-enum of~$\Ccal(\vec{R})$, contradicting our hypothesis.

On input~$n$, $\emptyset'$-effectively search for a~$\tau_n \succeq \sigma$
such that~$\Gamma^{\tau_n \oplus \emptyset'} \restr n$ is defined.
Such~$\tau_n$ exists as~$\sigma \prec G$ and~$\Gamma^{G \oplus \emptyset'}$ is total.
As~$\tau_n \not \in W_{bad}$, $[\Gamma^{\tau_n \oplus \emptyset'}] \cap \Ccal(\vec{R}) \neq \emptyset$
and therefore~$(\tau_n : n \in \omega)$ is a $\emptyset'$-computable 1-enum of~$\Ccal(\vec{R})$.
\end{proof}

Note that if we assume that~$G$ is weakly 3-generic and therefore avoids the set
$W_{bad} \cup W_{partial}$ where
\[
W_{partial} = \{ \sigma \in 2^{<\omega} : (\forall \tau \succeq \sigma) |\Gamma^{\tau \oplus \emptyset'}| < |\sigma| \}
\]
then we can furthermore impose that~$\tau_{n+1} \succeq \tau_n$ and~$\emptyset'$-compute a member of~$\Ccal(\vec{R})$.
This suffices to reprove that no weakly 3-generic helps solving an unsolvable intance of~$\coh$.

We now prove a theorem inspired by the proof
of domination closure of p-cohesive degrees by Jockusch and Stephan~\cite{Jockusch1993cohesive}.

\begin{theorem}
For any uniformly computable sequence of sets~$R_0, R_1, \dots$
such that~$\Ccal(\vec{R})$ has no~$\emptyset'$-computable 1-enum,
every~$\vec{R}$-cohesive set is of hyperimmune degree.
\end{theorem}
\begin{proof}
Suppose for the contradiction that there exists some~$\vec{R}$-cohesive set~$C = \{a_0 <  a_1 < \dots \}$
and a computable set~$B = \{b_0 < b_1 < \dots\}$ such that~$(\forall i)(a_i < b_i)$.
For each~$n \in \omega$, let $B_n = \{n, n+1, \dots, b_n\}$. Note that~$a_n \in B_n$ for every~$n$,
and therefore for every length~$s$, there exists a string~$\sigma_s$ of length~$s$
such that~$(\exists b)(\forall n > b) R_{\sigma_s} \cap B_n \neq \emptyset$.
Let~$\sigma_0, \sigma_1, \dots$ be the~$\emptyset'$-computable sequence of such strings.
We claim that this sequence is a 1-enum of~$\Ccal(\vec{R})$, therefore contradicting our hypothesis.
Indeed, as $(\exists b)(\forall n > b) R_{\sigma_s} \cap B_n \neq \emptyset$, the set~$R_{\sigma_s}$
is infinite and therefore~$\Ccal(\vec{R}) \cap [\sigma_s] \neq \emptyset$.
\end{proof}

Of course, there exists some uniformly computable sequence of sets~$R_0, R_1, \dots$
with no computable $\vec{R}$-cohesive set but with an~$\vec{R}$-cohesive
set of hyperimmune-free degree. Simply apply Lemma~\ref{lem:sigma3-coh} with~$\Ecal = \{X\}$
where~$X$ is a~$\Delta^0_3$ set of hyperimmune-free degree. Such a set is known
to exists by Miller and Martin~\cite{Miller1968degrees}. The class $\Ecal$ has no $\emptyset'$-computable
member as every $\Delta^0_2$ set is hyperimmune.

\subsection{How randomness helps solving cohesiveness}\label{sect:coh-randomness}

We now explore the interrelations between cohesiveness
and the measure-theoretic paradigm of typicality, namely, algorithmic randomness.

\begin{definition}[Randomness]
A~\emph{$\Sigma^0_n$ (Martin-Löf) test} is a sequence~$U_0, U_1, \dots$ of uniformly~$\Sigma^0_n$
classes such that~$\mu(U_i) \leq 2^{-i}$ for every~$i \in \omega$.
A real~$Z$ is \emph{$n$-random} if for every~$\Sigma^0_n$ test $U_0, U_1, \dots$,
$Z \not \in \bigcap_i U_i$.
A real~$Z$ is \emph{weakly $n$-random} if it is in every~$\Sigma^0_n$ class of measure~1.
\end{definition}

We shall say \emph{Martin-L\"of random} for \emph{1-random}.
Wang~\cite{Wang2013Omitting} proved that whenever a uniformly computable sequence of sets~$R_0, R_1, \dots$
has no computable~$\vec{R}$-cohesive sets, there exists a Martin-Löf random real computing no~$\vec{R}$-cohesive set.
Thanks to Corollary~\ref{cor:low-helps-not-coh}, 
we know that it suffices to take any low Martin-Löf random real to obtain this property.
Wang asked whether we can always ensure the existence of a 3-random real computing
an~$\vec{R}$-cohesive set whenever the instance is unsolvable. 
The next two lemmas answer this question by proving that it depends on the considered
sequence of sets~$\vec{R}$.

\begin{lemma}
There exists a uniformly computable
sequence of sets~$R_0, R_1, \dots$ with no computable~$\vec{R}$-cohesive set,
but such that every 2-random real computes an~$\vec{R}$-cohesive set.
\end{lemma}
\begin{proof}
Let~$\Dcal$ be a $\Pi^{0,\emptyset'}_1$ class of positive measure
with no~$\emptyset'$-computable member.
By Lemma~\ref{lem:tree-coh-inv},
there exists a uniformly computable sequence of sets~$R_0, R_1, \dots$
such that~$\Ccal(\vec{R}) = \Dcal$.
By Kautz~\cite{Kautz1991Degrees,Kautz1998improved}, every 2-random real is,
up to prefix, a member of~$\Ccal(\vec{R})$.
Therefore, by Lemma~\ref{lem:coh-tree}, every 2-random real
computes an~$\vec{R}$-cohesive set.
\end{proof}

\begin{lemma}\label{lem:randomness-p-cohesive}
For every $n \geq 3$, no (weakly) $n$-random real computes a p-cohesive set.
\end{lemma}
\begin{proof}
Jockusch and Stephan~\cite{Jockusch1993cohesive} proved that degrees of p-cohesive sets
are those whose jump is PA relative to~$\emptyset'$.
By a relativization of~Stephan~\cite{Stephan2006Martin}, every 2-random real
whose jump is of PA degree relative to~$\emptyset'$ is high.
By Kautz~\cite{Kautz1991Degrees}, no weakly 3-random real is high.
For every $n \geq 3$, every (weakly) $n$-random real is a weakly 3-random real.
\end{proof}

Avigad et al.~\cite{Avigad2012Algorithmic}
introduced the principle $\wwkls{n}$ stating that every~$\Delta^0_n$ tree of positive measure has a path.
In particular, $\wwkls{1}$ is~$\wwkl$.  
Thanks to Lemma~\ref{lem:randomness-p-cohesive},
for every~$n \in \omega$, one can apply the usual constructions to build an~$\omega$-model
of~$\wwkls{n}$ which does not contain any p-cohesive set and therefore is not a model of~$\coh$.
Pick any~$n$-random $Z$ which does not compute any p-cohesive set
and consider it as an infinite join~$Z_0 \oplus Z_1 \oplus \dots$.
By Van Lambalgen's theorem~\cite{VanLambalgen1990axiomatization}, the~$\omega$-structure whose second-order part
is the Turing ideal~$\{ X : (\exists i) X \leq_T Z_0 \oplus \dots \oplus Z_i \}$
is a model of~$\wwkls{n}$. Moreover it does not contain a p-cohesive set.

\subsection{How Ramsey-type theorems help solving cohesiveness}\label{sect:coh-ramsey-type}

In his paper separating Ramsey's theorem for pairs from
weak K\"onig's lemma, Liu~\cite{Liu2012RT22} proved that every (non-necessarily effective)
set~$A$ has an infinite subset of either it or its complement which is not of PA degree.
The absence of effectiveness conditions on~$A$ shows the combinatorial nature
of the weakness of the infinite pigeonhole principle.
On the other hand, the author~\cite{PateyCombinatorial} showed that this weakness depends on the choice
of the instance of~$\wkl$, by constructing a computable tree with no computable path together
with a~$\Delta^0_2$ set~$A$ such that every infinite subset of either~$A$ or~$\overline{A}$
computes a path trough the tree. We answer a similar question for cohesiveness
and study the weakness of the pigeonhole principle for typical partitions.

\begin{lemma}
There exists a~$\Delta^0_3$ (in fact low over~$\emptyset'$) set~$A$ 
and a uniformly computable sequence of sets~$R_0, R_1, \dots$
with no computable~$\vec{R}$-cohesive set, such that
every infinite subset of either~$A$ or~$\overline{A}$
computes an~$\vec{R}$-cohesive set.
\end{lemma}
\begin{proof}
Fix a set~$A$ which is low over~$\emptyset'$ and bi-immune relative to~$\emptyset'$.
The set of the infinite, increasing sequences which form an subset of either~$A$ or~$\overline{A}$ is $\Pi^{0,A}_1$,
hence~$\Pi^{0,\emptyset'}_2$ in the Baire space:
$$
\Ecal = \{ X \in \omega^\omega : (\forall s)[X(s) < X(s+1)] 
	\wedge [(\forall s)(X(s) \in A) \vee (\forall s)(X(s) \in \overline{A})] \}
$$
Moreover, $\Ecal$ has no~$\emptyset'$-computable member by bi-immunity relative to~$\emptyset'$ of~$A$.
Apply Lemma~\ref{lem:sigma3-coh} to complete the proof.
\end{proof}

In a previous section, we constructed a uniformly computable
sequence of sets~$R_0, R_1, \dots$ with no computable $\vec{R}$-cohesive set
such that every 2-random real computes an~$\vec{R}$-cohesive set.
The following lemma strengthens this result by constructing an unsolvable instance of~$\coh$
solvable by every infinite subset of any 2-random real.

\begin{definition}[Diagonal non-computability]
A function~$f : \omega \to \omega$ is \emph{diagonaly non-computable} relative to~$X$
if for every~$e \in \omega$, $f(e) \neq \Phi^X_e(e)$.
\end{definition}

By Kjos-Hanssen~\cite{Kjos-Hanssen2009Infinite} and Greenberg and Miller~\cite{Greenberg2009Lowness}, 
a set computes a function d.n.c.\ relative
to~$\emptyset^{(n-1)}$ if and only if it computes an infinite subset of an~$n$-random.

\begin{lemma}
There exists a uniformly computable sequence of sets~$R_0, R_1, \dots$
with no computable~$\vec{R}$-cohesive set, such that
every function d.n.c.\ relative to~$\emptyset'$ computes an~$\vec{R}$-cohesive set.
\end{lemma}
\begin{proof}
The class of functions which are d.n.c.\ relative to~$\emptyset'$ is $\Pi^{0, \emptyset'}_1$ in the Baire space:
$$
\Ecal = \left\{ f \in \omega^\omega : (\forall e)[\Phi^{\emptyset'}_e(e) \uparrow \vee f(e) \neq \Phi^{\emptyset'}_e(e)]\right\}
$$
Moreover, $\Ecal$ has no~$\emptyset'$-computable member.
Apply Lemma~\ref{lem:sigma3-coh} to complete the proof.
\end{proof}

In contrast with this lemma, 
if we require a bit more uncomputability in the~$\vec{R}$-cohesive sets of the sequence~$R_0, R_1, \dots$,
we can ensure the existence of a function d.n.c.\ relative to~$\emptyset'$
which does not compute an~$\vec{R}$-cohesive set.

\begin{theorem}\label{thm:coh-dnc-avoid}
Fix a uniformly computable sequence of sets~$R_0, R_1, \dots$
such that~$\Ccal(\vec{R})$ has no $\emptyset'$-computable 1-enum.
For every set~$X$, there exists a function~$f$ d.n.c.\ relative to~$X$
whose jump does not compute a 1-enum of~$\Ccal(\vec{R})$.
In particular, $f$ does not compute an~$\vec{R}$-cohesive set.
\end{theorem}

The proof of Theorem~\ref{thm:coh-dnc-avoid} is done by a bushy tree forcing argument.
See the survey from Khan and Miller~\cite{Khan2014Forcing} for terminology and definitions.
Fix a set~$X$. We will construct a GL${}_1$ function which is d.n.c.\ relative to~$X$.
Our forcing conditions are tuples~$(\sigma, B)$ where~$\sigma \in \omega^{<\omega}$
and $B \subseteq \omega^{<\omega}$ is an upward-closed set $k$-small above~$\sigma$ for some~$k \in \omega$.
A sequence~$f$ \emph{satisfies} a condition~$(\sigma, B)$ if $\sigma \prec f$ and $B$ is small above every initial segment of~$f$.
Our initial condition is~$(\varepsilon, B_{DNC}^X)$ where~
$$
B_{DNC}^X = \{ \sigma \in \omega^{<\omega} : (\exists e) \sigma(e) = \Phi^X_e(e) \}
$$
Therefore every infinite sequence~$f$ satisfying~$(\varepsilon, B_{DNC}^X)$ is d.n.c.\ relative to~$X$.
Thanks to the following lemma, we can prevent~$f \oplus \emptyset'$ from computing
a 1-enum of~$\Ccal(\vec{R})$. As the constructed function~$f$ is GL${}_1$, $f' \leq_T f \oplus \emptyset'$
does not compute a 1-enum of~$\Ccal(\vec{R})$.

\begin{lemma}
For every condition~$c = (\sigma, B)$ and every Turing functional~$\Gamma$,
there exists an extension~$d = (\tau, C)$ forcing~$\Gamma^{f \oplus \emptyset'}$ to be partial
or such that~$\Gamma^{\tau \oplus \emptyset'}$ is not a 1-enum of~$\Ccal(\vec{R})$.
\end{lemma}
\begin{proof}
Suppose that~$B$ is $k$-small above~$\sigma$.
For every~$n \in \omega$,
define the~$\Sigma^{0, \emptyset'}_1$ set~$D_n = \{ \tau \in \omega^{<\omega} : \Gamma^{\tau \oplus \emptyset'}(n) \downarrow \in 2^n\}$.
Make a~$\emptyset'$-effective search for an~$n \in \omega$ such that one of the following holds:
\begin{itemize}
	\item[(a)] $D_n$ is $k2^n$-small above~$\sigma$ for some~$n \in \omega$
	\item[(b)] $D_{n,\rho} = \{ \tau \in \omega^{<\omega} : \Gamma^{\tau \oplus \emptyset'}(n) \downarrow = \rho \}$
	is $k$-big above~$\sigma$ for some string~$\rho \in 2^n$ such that $[\rho] \cap \Ccal(\vec{R}) = \emptyset$.
\end{itemize}
Such~$n$ exists, as otherwise, for every~$n \in \omega$, $D_n$ is $k2^n$-big above~$\sigma$.
By the smallness additivity property, $D_{n,\rho}$ is $k$-big above~$\sigma$ for some~$\rho \in 2^n$.
For every such string~$\rho$, $[\rho] \cap \Ccal(\vec{R}) \neq \emptyset$. Therefore we can $\emptyset'$-compute
a 1-enum of~$\Ccal(\vec{R})$ by searching on each input~$n$ for some~$\rho$ of length~$n$
such that~$D_{n,\rho}$ is $k$-big above~$\sigma$.

If we are in case~(a), take~$d = (\tau, C \cup D_n)$ as the desired extension.
The condition~$d$ forces~$\Gamma^{f \oplus \emptyset'}$ to be partial.
If we are in case~(b), by the concatenation property, there exists an extension~$\tau \in D_{n,\rho}$
such that~$B$ is still~$k$-small above~$\tau$. The condition $d = (\tau, B)$ is an extension
forcing~$\Gamma^{f \oplus \emptyset'}$ not to be a 1-enum of~$\Ccal(\vec{R})$
as~$\Gamma^{f \oplus \emptyset'}(n) = \Gamma^{\tau \oplus \emptyset'}(n) = \rho$ and~$[\rho] \cap \Ccal(\vec{R}) = \emptyset$.
\end{proof}

Looking at the proof of the previous lemma, we can~$\emptyset'$-decide in which case we are, and then use the knowledge of~$f$
to see which path has been chosen in the bushy tree. The construction therefore yields a~GL${}_1$ sequence.

\section{Ramsey's theorem and computable reducibility}\label{sect:rt-computable-reduc}

The strength of Ramsey's theorem is known to remain the same when
changing the number of colors in the setting of reverse mathematics.
Indeed, given some coloring~$f : [\omega]^n \to k^2$, we can define another coloring
$g : [\omega]^n \to k$ by merging colors together by blocks of size~$k$. After one application of~$\rt^n_k$
to the coloring~$g$, we obtain an infinite set~$H$ over which~$f$ uses at most~$k$ different colors.
Another application of~$\rt^n_k$ gives an infinite $f$-homogeneous set.
This standard proof of~$\rca \vdash \rt^n_k \imp \rt^n_{k^2}$ involves two applications
of~$\rt^n_k$. In this section, we show that in the computable reducibility setting, 
multiple applications are really necessary
to reduce~$\rt^n_k$ to~$\rt^n_\ell$ whenever~$k > \ell$ and~$n \geq 2$.

Note that two applications of~$\rt^n_2$ are sufficient to deduce~$\rt^n_k$ in the case~$n \geq 4$,
as Jockusch~\cite{Jockusch1972Ramseys} proved that every computable instance of~$\rt^n_k$ has a~$\Pi^0_n$ solution,
and that for every set~$X$, there exists an~$X$-computable instance of~$\rt^n_2$
such that every solution computes~$X^{(n-2)}$.

\subsection{Cohesiveness and strong reducibility}\label{sect:coh-strong-reduc}

We start our analysis with partitions of integers.
Of course, every computable partition has an infinite computable homogeneous set,
so we need to consider non-effective partitions and strong computable reducibility.
The study of~$\rt^1_k$ over strong reducibility has close connections with
cohesiveness.
Dzhafarov~\cite{Dzhafarov2014Cohesive} proved that~$\coh \not \leq_{sc} \mathsf{D}^2_{<\infty}$ by iterating the following theorem.

\begin{theorem}[Dzhafarov~\cite{Dzhafarov2014Cohesive}]\label{thm:dzhafarov-partitions}
For every~$k \geq 2$ and $\ell < 2^k$, there is a finite sequence~$R_0, \dots, R_{k-1}$
such that for all partitions~$A_0 \cup \dots \cup A_{\ell-1} = \omega$ hyperarithmetical in~$\vec{R}$,
there is an infinite subset of some~$A_j$ that computes no~$\vec{R}$-cohesive set.
\end{theorem}

Hirschfeldt and Jockusch noticed in~\cite{Hirschfeldtnotions} that the proof of Theorem~\ref{thm:dzhafarov-partitions}
can be slightly modified to obtain a proof that~$\rt^1_k \not \leq_{sc} \rt^1_\ell$ whenever~$k > \ell \geq 2$.
Mont\'alban asked whether the hyperarithmetic effectiveness restriction
can be removed from Dzhafarov's theorem. 
We give a positive answer, which has been proved independently by Hirschfeldt and Jockusch~\cite{Hirschfeldtnotions}.
Moreover, we show that~$\vec{R}$ can be chosen to be low. 

Given two integers~$k, \ell \geq 1$, we let~$\pi(k, \ell)$ denote the unique~$a \geq 1$ such that
$k = a \cdot \ell - b$ for some~$b \in [0, \ell)$. Informally, $\pi(k, \ell)$ is the minimal number of pigeons
we can ensure in at least one pigeonhole, given~$k$ pigeons and~$\ell$ pigeonholes.
In particular, $\pi(k, \ell) \geq 2$ whenever $k > \ell \geq 1$.
We prove the following theorem, from which
we deduce several corollaries about cohesiveness and~$\rt^1_k$.

\begin{theorem}\label{thm:rt1-hyperimmunity}
Fix some~$k \geq 1$ and $\ell \geq 2$, some set~$I$ and a sequence of $k$ $I$-hyperimmune sets~$B_0, \dots, B_{k-1}$.
For every~$\ell$-partition~$A_0 \cup \dots \cup A_{\ell-1} = \omega$,
there exists an infinite subset~$H$ of some~$A_i$ such
that $\pi(k, \ell)$ sets among the~$B$'s are $I \oplus H$-hyperimmune.
\end{theorem}

We will postpone the proof of Theorem~\ref{thm:rt1-hyperimmunity} until after Corollary~\ref{cor:srt2-not-reduc-colors}.
Using the existence of a low
$k$-partition~$B_0 \cup \dots \cup B_{k-1} = \omega$ such that
$\overline{B_j}$ is hyperimmune for every~$j < k$,
we deduce the following corollary.

\begin{corollary}\label{cor:d2-general-colors}
For every~$k > \ell \geq 2$, there is a low $k$-partition~
$B_0 \cup \dots \cup B_{k-1} = \omega$
such that for all $\ell$-partitions~$A_0  \cup \dots \cup A_{\ell-1} = \omega$,
there is an infinite subset~$H$ of some~$A_i$ and a pair~$j_0 < j_1 < k$
such that every infinite~$H$-computable set intersects both~$B_{j_0}$ and~$B_{j_1}$.
\end{corollary}
\begin{proof}
Fix some~$k > \ell \geq 2$ and a low $k$-partition $B_0 \cup \dots \cup B_{k-1} = \omega$
such that $\overline{B_j}$ is hyperimmune for every~$j < k$.
Since~$k > \ell \geq 2$, $\pi(k, \ell) \geq 2$.
Therefore, by Theorem~\ref{thm:rt1-hyperimmunity}, for every $\ell$-partition
$A_0  \cup \dots \cup A_{\ell-1} = \omega$, there is an infinite subset~$H$ of some~$A_i$ and a pair~$j_0 < j_1 < k$
such that~$\overline{B_{j_0}}$ and~$\overline{B_{j_1}}$ are~$H$-hyperimmune.
In particular, every infinite $H$-computable set intersects both~$B_{j_0}$ and~$B_{j_1}$.
\end{proof}

The positive answer to Mont\'alban's question is an immediate consequence
of the previous corollary.

\begin{corollary}\label{cor:dzhafarov-comb}
For every~$k \geq 2$ and $\ell < 2^k$, there is a finite sequence
of low sets~$R_0, \dots, R_{k-1}$
such that for all partitions~$A_0 \cup \dots \cup A_{\ell-1} = \omega$,
there is an infinite subset of some~$A_i$ that computes no~$\vec{R}$-cohesive set.
\end{corollary}
\begin{proof}
Given~$k \geq 2$ and~$\ell < 2^k$, fix the low $2^k$-partition $(B_\sigma : \sigma \in 2^k)$
whose existence is stated by Corollary~\ref{cor:d2-general-colors}.
For each~$i < k$, define~$R_i = \bigcup_{\sigma(i) = 1} B_\sigma$.
Note that by disjointness of the $B$'s, $\overline{R_i} = \bigcup_{\sigma(i) = 0} B_\sigma$.
By choice of the $B$'s, for all $\ell$-partitions~$A_0 \cup \dots \cup A_{\ell-1} = \omega$, there is
an infinite subset~$H$ of some~$A_j$ and a pair~$\sigma <_{lex} \tau \in 2^k$
such that every infinite~$H$-computable set intersects both~$B_\sigma$ and~$B_\tau$.
Let~$i < k$ be the least bit such that~$\sigma(i) \neq \tau(i)$.
As~$\sigma <_{lex} \tau$, $\sigma(i) = 0$ and~$\tau(i) = 1$. By definition of~$R_i$,
$B_{\tau} \subseteq R_i$ and~$B_{\sigma} \subseteq \overline{R_i}$. 
Therefore no infinite~$H$-computable set is homogeneous for~$R_i$. In particular 
no infinite~$H$-computable set is~$\vec{R}$-cohesive.
\end{proof}

The construction of the~$B$'s is done uniformly in~$k$.
We can therefore deduce the following corollary.

\begin{corollary}
There exists a sequence of low sets~$R_0, R_1, \dots$
such that every finite partition of~$\omega$ has an infinite subset in one of its parts
which does not compute an~$\vec{R}$-cohesive set.
\end{corollary}

The effectiveness of~$B$ in the statement of Corollary~\ref{cor:d2-general-colors} enables us
to deduce computable non-reducibility results about stable Ramsey's theorem for pairs,
thanks to the computable equivalence between~$\srt^2_\ell$ and the statement~$\mathsf{D}^2_\ell$.

\begin{corollary}\label{cor:srt2-not-reduc-colors}
For every~$k > \ell \geq 2$, $\srt^2_k \not \leq_c \srt^2_\ell$.
\end{corollary}
\begin{proof}
Fix $k > \ell \geq 2$. By Corollary~\ref{cor:d2-general-colors},
there is a~$\Delta^0_2$ $k$-partition~
$B_0 \cup \dots \cup B_{k-1} = \omega$
such that for all $\ell$-partitions~$A_0, \dots, A_{\ell-1}$ of~$\omega$,
there is an infinite subset~$H$ of some~$A_i$ which does not compute an infinite
subset of any~$B_j$.
By Cholak et al.~\cite{Cholak2001strength}, for every stable computable function~$f : [\omega]^2 \to \ell$,
there exists a~$\Delta^0_2$~$\ell$-partition~$A_0 \cup \dots \cup A_{\ell-1} = \omega$ such that every
infinite subset of a part computes an infinite~$f$-homogeneous set. Therefore, for every such function~$f$,
there exists an infinite $f$-homogeneous set which does not compute an infinite subset of any~$B_j$.
By Schoenfield's limit lemma~\cite{Shoenfield1959degrees}, the~$\Delta^0_2$ approximation~$g : [\omega]^2 \to k$ of the~$k$-partition~
$B_0 \cup \dots \cup B_{k-1} = \omega$ is a stable computable function and every infinite~$g$-homogeneous set with color~$j$ is an infinite subset of~$B_j$.
\end{proof}

We now turn to the proof of Theorem~\ref{thm:rt1-hyperimmunity}.
We shall prove it by induction over~$\ell$, using a forcing construction whose
forcing conditions are Mathias conditions~$(F, X)$ where~$X$ is an infinite set 
such that the~$B$'s are $X \oplus I$-hyperimmune. The case where~$\ell = 1$ trivially holds
since $\pi(k, 1) = k$.

\subsubsection{Forcing limitlessness}

For every~$\ell$-partition~$A_0 \cup \dots \cup A_{\ell-1} = \omega$, 
we want to satisfy the following scheme of requirements to ensure that~$G \cap A_i$
is infinite for each~$i < \ell$.
\[
\Qcal_p : (\exists n_0, \dots, n_{\ell-1} > p)[n_0 \in G \cap A_0 \wedge \dots \wedge n_{\ell-1} \in G \cap A_{\ell-1}]
\]
Of course, all requirements may not be satisfiable
if some part~$A_i$ is finite. Usually, a forcing argument starts with the assumption that the instance 
is non-trivial, that is, does not admit a solution with the desired properties (cone avoiding, low, ...).
In order to force the solution to be infinite,
it suffices to ensure that the reservoirs satisfy the desired properties, and therefore 
cannot be a solution to a non-trivial instance.

In our case, we say that an~$\ell$-partition~$A_0 \cup \dots \cup A_{\ell-1}$ is~\emph{non-trivial} if
there is no infinite set~$H$ included in the complement of one of the~$A$'s and such that
the $B$'s are $H \oplus I$-hyperimmune. 
The following lemma states that we can focus on non-trivial partitions without loss of generality.

\begin{lemma}\label{lem:colors-ramsey-partition-reduction}
For every trivial $\ell$-partition $A_0 \cup \dots \cup A_{\ell-1}$, 
there is an infinite set~$H \subseteq A_i$ for some~$i < \ell$ such that $\pi(k, \ell)$ sets among the $B$'s are $H \oplus I$-hyperimmune.
\end{lemma}
\begin{proof}
Let~$G = \{n_0 < n_1 < \dots \}$ be an infinite subset of $\overline{A}_i$ for some~$i < \ell$
such that the~$B$'s are $G \oplus I$-hyperimmune.
Define the~$(\ell-1)$-partition~$(C_j : j \neq i)$ by setting~$C_j = \{ s \in \omega : n_s \in A_j \}$ for each~$j \neq i$.
By induction hypothesis, there exists an infinite set~$H_0 \subseteq C_j$ for some~$j \neq i$
such that $\pi(k, \ell - 1)$ sets among the $B$'s are $H_0 \oplus G \oplus I$-hyperimmune.
Note that $\pi(k, \ell-1) \leq \pi(k, \ell)$.
The set $H = \{ n_s : s \in H_0 \}$ is an $H_0 \oplus G$-computable subset of $A_j$
and $\pi(k, \ell)$ sets among the $B$'s are $H \oplus J$-hyperimmune.
\end{proof}

Notice that the proof of Lemma~\ref{lem:colors-ramsey-partition-reduction}
uses the induction hypothesis with a different context, namely, $G \oplus I$ instead of $I$.
This is where we needed to use the relativized version of the theorem in the proof.
A condition~$c = (F, X)$ \emph{forces $\Qcal_p$}
if there exists some~$n_0, \dots, n_{m-1} > p$ such that~$n_i \in F \cap A_j$
for each~$i < \ell$. Therefore, if~$G$ satisfies~$c$ and~$c$ forces~$\Qcal_p$,
then~$G$ satisfies the requirement~$\Qcal_p$.
We now prove that the set of conditions forcing~$\Qcal_p$ is dense for each~$p \in \omega$.
Thus, every sufficiently generic filter will induce an infinite solution.

\begin{lemma}\label{lem:colors-ramsey-coh-strong-reduc-force-Q}
For every condition~$c$ and every~$p \in \omega$, there is an extension forcing~$\Qcal_p$.
\end{lemma}
\begin{proof}
Fix some~$p \in \omega$. It is sufficient to show that given a condition~
$c = (F, X)$ and some~$i < \ell$,
there exists an extension~$d_0 = (E, Y)$ and some integer~$n_i > p$
such that~$n_i \in E \cap A_i$.
By iterating the process for each~$i < \ell$, we obtain the desired extension~$d$.
By definition of non-triviality, $A_i$ is co-immune in~$X$
and therefore $X \cap A_i$ is infinite.
Take any~$n_i \in X \cap A_i \cap (p, +\infty)$.
The condition $d_0 = (F \cup \{n_i\}, X \setminus [0, n_i])$ is the desired extension.
\end{proof}

\subsubsection{Forcing non-homogeneity}

The second scheme of requirements aims at ensuring that
for some~$i < \ell$,
at least~$\pi(k, \ell)$ sets among the~$B$'s are $(G \cap A_i) \oplus I$-hyperimmune.
The requirements are of the following form for each~$j < k$
and each tuple of indices~$\vec{e} = e_0, \dots, e_{\ell-1}$.
\[
\Rcal_{\vec{e}, j} : \hspace{10pt} \Rcal^{A_0, B_j}_{e_0} \vee \dots \vee \Rcal^{A_{\ell-1}, B_j}_{e_{\ell-1}}
\]
where $\Rcal_e^{A, B}$ is the statement ``$\Phi_e^{(G \cap A) \oplus I}$ does not dominate $p_B$''.

We claim that if all the requirements are satisfied, then~$(G \cap A_i)$ has the desired property
for some~$i < \ell$. Indeed, if for some fixed~$j < k$, all the requirements~$\Rcal_{\vec{e}, j}$
are satisfied, then by the usual pairing argument, there is some~$i < \ell$ such that $B_j$
is $(G \cap A_i) \oplus I$-hyperimmune. So if all the requirements are satisfied, then by the pigeonhole
principle, there is some~$i < \ell$ such that~$\pi(k, \ell)$ sets among the~$B$'s are $(G \cap A_i) \oplus I$-hyperimmune.

A condition~\emph{forces $\Rcal_{\vec{e}, j}$} if every set~$G$ satisfying this condition also satisfies the requirement~$\Rcal_{\vec{e}}$.
The following lemma is the core of the forcing argument.

\begin{lemma}\label{lem:colors-ramsey-coh-strong-reduc-force-R}
For every condition~$c = (F, X)$, every~$j < k$ and every tuple of Turing indices~$\vec{e}$, 
there exists an extension~$d = (E, Y)$ forcing~$\Phi_{e_i}^{(G \cap A_i) \oplus I}$ not to dominate~$p_{B_j}$
for some~$i < \ell$.
\end{lemma}
\begin{proof}
Let~$f$ be the partial~$X \oplus I$-computable function which on input~$x$, 
searches for a finite set of integers~$U$ such that for every $\ell$-partition
$Z_0 \cup \dots \cup Z_{\ell-1} = X$, there is some~$i < \ell$
and some set~$E \subseteq Z_i$ such that $\Phi_{e_i}^{((F \cap A_i) \cup E) \oplus I}(x) \downarrow \in U$.
If such a set~$U$ is found, then~$f(x) = max(U)+1$, otherwise~$f(x) \uparrow$. We have two cases.
\begin{itemize}
	\item Case 1: The function~$f$ is total. By~$X \oplus I$-hyperimmunity of~$B_j$,
	$f(x) \leq p_{B_j}(x)$ for some~$x$. Let~$U$ be the finite set witnessing~$f(x) \downarrow$.
	Letting~$Z_i = X \cap A_i$ for each~$i < \ell$, there is some~$i$ and some finite set $E \subseteq X \cap A_i$
	such that~$\Phi_{e_i}^{((F \cap A_i) \cup E) \oplus I}(x) \downarrow \in U$.
	The condition~$d = (F \cup E, X \setminus [0, max(E)])$ is an extension forcing
	$\Phi_{e_i}^{(G \cap A_i) \oplus I}(x) < f(x) \leq p_{B_j}(x)$.

	\item Case 2: There is some~$x$ such that~$f(x) \uparrow$.
	By compactness, the~$\Pi^{0,X \oplus I}_1$ class $\Ccal$ of sets~$Z_0 \oplus \dots \oplus Z_{\ell-1}$
	such that~$Z_0 \cup \dots \cup Z_{\ell-1} = X$ and for every~$i < \ell$
	and every set~$E \subseteq Z_i$, $\Phi_{e_i}^{((F \cap A_i) \cup E) \oplus I}(x) \uparrow$
	is non-empty. By the hyperimmune-free basis theorem~\cite{Jockusch197201},
	there is some~$\ell$-partition $Z_0 \oplus \dots \oplus Z_{\ell-1} \in \Ccal$
	such that all the~$B$'s are $Z_0 \oplus \dots \oplus Z_{\ell-1} \oplus X \oplus I$-hyperimmune.
	Let~$i < \ell$ be such that~$Z_i$ is infinite. The condition~$d = (F, Z_i)$
	is an extension of~$c$ forcing~$\Phi_{e_i}^{(G \cap A_i) \oplus I}(x) \uparrow$.
\end{itemize}
\end{proof}

\subsubsection{Construction}

We have all necessary ingredients to build an infinite set~$G$
such that each~$G \cap A_i$ is infinite, and such that $\pi(k, \ell)$ sets among the~$B$'s
are $(G \cap A_i) \oplus I$-hyperimmune for some~$i < \ell$. 
Thanks to Lemma~\ref{lem:colors-ramsey-coh-strong-reduc-force-Q} 
and Lemma~\ref{lem:colors-ramsey-coh-strong-reduc-force-R}, define an infinite descending
sequence of conditions~$(\varepsilon, \omega) \geq c_0 \geq \dots$ such that for each~$s \in \omega$,
\begin{itemize}
	\item[(a)] $c_s$ forces~$\Qcal_s$
	\item[(b)] $c_s$ forces~$\Rcal_{\vec{e},j}$ if~$s = \tuple{\vec{e},j}$
\end{itemize}
where~$c_s = (F_s, X_s)$.
Define the set~$G = \bigcup_s F_s$. By (a), $G \cap A_i$ is infinite 
for every~$i < \ell$, and by (b), each requirement~$\Rcal_{\vec{e},j}$
is satisfied. This finishes the proof of Theorem~\ref{thm:rt1-hyperimmunity}.

\subsection{Reducibility to Ramsey's theorem for pairs}\label{sect:rt-reducibility}

Dorais et al.~\cite{Dorais2016uniform} asked whether $\rt^n_k \not \leq_c \rt^n_\ell$ for every~$n \geq 2$ and $k > \ell \geq 2$.
Hirschfeldt and Jockusch~\cite{Hirschfeldtnotions} 
and Rakotoniaina~\cite{Rakotoniaina2015Computational} proved that~$\srt^n_k$ is not uniformly reducible
to~$\rt^n_\ell$ whenever~$k > \ell$. We extend the result to computable reducibility.
In the first place, we shall focus on the case~$n = 2$.
For this, we will take advantage of the proof of~$\rt^2_\ell$ that applies the cohesiveness principle
to obtain a stable coloring~$f : [\omega]^2 \to \ell$. This coloring can itself be considered
as the~$\Delta^0_2$ approximation of a~$\emptyset'$-computable $\ell$-partition of~$\omega$,
and therefore as a non-effective instance of~$\rt^1_\ell$.
Any infinite subset of one of its parts computes an infinite set homogeneous for~$f$.

In the previous section, we have shown how to diagonalize against every~$\ell$-partition,
simply using the fact that the complement of the parts of the instance of~$\rt^n_k$
are hyperimmune. The author proved in~\cite{Patey2015Iterative} that $\coh$ instances
admit solutions preserving the hyperimmunity of a predefined collection of hyperimmune sets.

\begin{theorem}[Patey~\cite{Patey2015Iterative}]\label{thm:patey-coh-preserving-hyperimmunities}
For every sequence of hyperimmune sets~$A_0, A_1, \dots$
and every uniformly computable sequence of sets~$R_0, R_1, \dots$,
there is an infinite $\vec{R}$-cohesive set~$C$ such that the~$A$'s are 
hyperimmune relative to~$C$.
\end{theorem}

Note that this theorem is optimal in the sense that every p-cohesive set is hyperimmune.
Using Theorem~\ref{thm:rt1-hyperimmunity}, we can deduce the following theorem.

\begin{theorem}\label{thm:rt2-hyperimmunity}
Fix some~$k \geq 1$ and $\ell \geq 2$, some set~$I$ and a sequence of $k$ $I$-hyperimmune sets~$B_0, \dots, B_{k-1}$.
Every $I$-computable coloring $f: [\omega]^2 \to \ell$
has an infinite $f$-homogeneous set~$H$ such
that $\pi(k, \ell)$ sets among the~$B$'s are $I \oplus H$-hyperimmune.
\end{theorem}
\begin{proof}
Fix $k$, a sequence of $I$-hyperimmune sets $B_0, \dots, B_{k-1}$ for some set~$I$.
Let~$f : [\omega]^2 \to \ell$ be an $I$-computable coloring
and consider the sequence of sets~$R_0, R_1, \dots$ defined for each~$x \in \omega$ by
\[
R_x = \{ s : f(x, s) = 1 \}
\]
By Theorem~\ref{thm:patey-coh-preserving-hyperimmunities}, there
is an infinite~$\vec{R}$-cohesive set~$C$ such that the~$B$'s are hyperimmune relative to~$C \oplus I$.
Let~$\tilde{f} : \omega \to \ell$ be defined by
$\tilde{f}(x) = \lim_{s \in C} f(x, s)$. By Theorem~\ref{thm:rt1-hyperimmunity}
relativized to~$C \oplus Z$, there is an infinite $\tilde{f}$-homogeneous set~$H$
such that $\pi(k, \ell)$ among the~$B$'s are $H \oplus C \oplus Z$-hyperimmune.
In particular, $H \oplus C \oplus I$ computes an infinite $f$-homogeneous set. 
\end{proof}

Using again the existence of a low $k$-partition $B_0 \cup \dots \cup B_{k-1}$
such that~$\overline{B_j}$ is hyperimmune
for every~$j < k$, we deduce the following corollary.

\begin{corollary}\label{cor:rt2-general-colors}
For every~$k > \ell \geq 2$, there is a low $k$-partition~
$B_0 \cup \dots \cup B_{k-1} = \omega$
such that each computable coloring $f: [\omega]^2 \to \ell$
has an infinite $f$-homogeneous set~$H$ and a pair~$j_0 < j_1 < k$
such that every infinite~$H$-computable set intersects both~$B_{j_0}$ and~$B_{j_1}$.
\end{corollary}
\begin{proof}
Fix some~$k > \ell \geq 2$ and a low $k$-partition $B_0 \cup \dots \cup B_{k-1} = \omega$
such that $\overline{B_j}$ is hyperimmune for every~$j < k$.
Since~$k > \ell \geq 2$, $\pi(k, \ell) \geq 2$.
Therefore, by Theorem~\ref{thm:rt2-hyperimmunity}, for every $\rt^2_\ell$-instance
$f : [\omega]^2 \to \ell$, there is an infinite $f$-homogeneous set~$H$ and a pair~$j_0 < j_1 < k$
such that~$\overline{B_{j_0}}$ and~$\overline{B_{j_1}}$ are~$H$-hyperimmune.
In particular, every infinite $H$-computable set intersects both~$B_{j_0}$ and~$B_{j_1}$.
\end{proof}

Using Corollary~\ref{cor:rt2-general-colors} in a relativized form, we can extend
the result to colorings over arbitrary tuples.

\begin{theorem}
For every~$n \geq 2$, and every $k > \ell \geq 2$, there is a~$\Delta^0_n$ $k$-partition~
$B_0 \cup \dots \cup B_{k-1} = \omega$
such that each computable coloring $f: [\omega]^n \to \ell$
has an infinite $f$-homogeneous set~$H$ and a pair~$j_0 < j_1 < k$
such that every infinite~$H$-computable set intersects both~$B_{j_0}$ and~$B_{j_1}$.
\end{theorem}
\begin{proof}
This is proved in a relativized form by induction over~$n$.
The case~$n = 2$ is proved by relativizing Corollary~\ref{cor:rt2-general-colors}.
Now assume it holds for some~$n$ in order to prove it for~$n+1$.
Let~$P \gg \emptyset^{(n-1)}$ be such that~$P' \leq \emptyset^{(n)}$.
Such a set exists by the relativized low basis theorem~\cite{Jockusch197201}.
Applying the induction hypothesis to~$P$,
there is a~$\Delta^{0,P}_2$ (hence~$\Delta^0_{n+1}$) $k$-partition~$B_0 \cup \dots \cup B_{k-1} = \omega$
such that each $P$-computable coloring $f: [\omega]^n \to \ell$
has an infinite $f$-homogeneous set~$H$ and a pair~$j_0 < j_1 < k$
such that every infinite~$H \oplus P$-computable set intersects both~$B_{j_0}$ and~$B_{j_1}$.

Let~$f : [\omega]^{n+1} \to \ell$ be a computable coloring.
By Jockusch~\cite[Lemma 5.4]{Jockusch1972Ramseys}, there exists an infinite set~$C$ pre-homogeneous for~$f$
such that~$C \leq_T P$.
(A set~$C$ is \emph{pre–homogeneous} if any two $(n+1)$-element subsets of~$C$
with the same first $n$ elements are assigned the same color by~$f$.)
Let~$\tilde{f} : [C]^n \to \ell$ be the~$P$-computable coloring defined for each~$\sigma \in [C]^n$
by~$\tilde{f}(\sigma) = f(\sigma, a)$, where~$a \in A$, $a > max(\sigma)$.
Every $\tilde{f}$-homogeneous set is~$f$-homogeneous.
By definition of~$B_0 \cup \dots \cup B_{k-1} = \omega$, 
there exists an infinite~$\tilde{f}$-homogeneous (hence $f$-homogeneous) set~$H$
and a pair~$j_0 < j_1 < k$
such that every infinite~$H \oplus P$-computable set intersects both~$B_{j_0}$ and~$B_{j_1}$.
\end{proof}

Using the fact that $\mathsf{D}^n_k \leq_c \srt^n_k$ for every $n, k \geq 2$, we obtain
the following corollary strengthening the result 
of Hirschfeldt and Jockusch~\cite{Hirschfeldtnotions} 
and Rakotoniaina~\cite{Rakotoniaina2015Computational}.

\begin{corollary}
For every~$n \geq 2$ and every~$k > \ell \geq 2$, $\srt^n_k \not \leq_c \rt^n_\ell$.
\end{corollary}

This answers in particular Question~7.1 of Dorais et al.~\cite{Dorais2016uniform}.
The following corollary answers positively Question~5.5.3 of Mileti~\cite{Mileti2004Partition}.

\begin{corollary}
There exists two stable computable functions~$f_1 : [\omega]^2 \to 2$
and~$f_2 : [\omega]^2 \to 2$ such that there is no computable~$g : [\omega]^2 \to 2$
with the property that every set~$H_g$ homogeneous for~$g$ computes
both a set~$H_{f_1}$ homogeneous for~$f_1$ and a set~$H_{f_2}$ homogeneous for~$f_2$.
\end{corollary}
\begin{proof}
By Corollary~\ref{cor:rt2-general-colors} with~$\ell = 2$ and~$k = 3$,
there exists a~$\Delta^0_2$ 3-partition~$B_0 \cup B_1 \cup B_2 = \omega$
such that each computable coloring $f: [\omega]^2 \to 2$
has an infinite $f$-homogeneous set~$H$ and a pair~$j_0 < j_1 < 3$
such that every infinite~$H$-computable set intersects both~$B_{j_0}$ and~$B_{j_1}$.
As in Corollary~\ref{cor:dzhafarov-comb}, we assume that the~$B$'s are disjoint.
By Schoenfield's limit lemma~\cite{Shoenfield1959degrees}, there exist
two stable computable colorings~$f_1$ and $f_2$ such that~$\lim_s f_1(\cdot, s) = B_0$ 
and $\lim_s f_2(\cdot, s) = B_1$.
If~$j_0 = 0$ (resp. $j_0 = 1$) then~$H$ does not compute an infinite set homogeneous for~$f_1$ (resp. $f_2$).
This completes the proof.
\end{proof}

\section{The weakness of free set and thin set theorems}

The combinatorics involved in our study of the free set and thin set theorems
differ deeply from our analysis of Ramsey's theorem in the previous sections.
Let~$\rt^n_{\ell,d}$ be the statement ``Every 
coloring~$f : [\omega]^n \to \ell$ has an infinite set~$H$ on which~$f$ uses
at most~$d$ colors.''
An analysis of the thin set theorems in the continuity of the previous sections
would consists in considering the computable
reductions between~$\rt^n_{\ell,d}$ and~$\rt^n_{k,d}$ whenever~$\ell < k$
for a fixed parameter~$d$. In this section, we consider the variation of the parameter~$d$,
and show that different~$d$'s lead to different subsystems of second-order arithmetic.

\subsection{Thin set theorem and strong reducibility}\label{sect:ts-strong-reduc}

We start our analysis with partitions of integers like we did with Ramsey's theorem. 
Every computable partition has an infinite 
computable set avoiding one of its parts. The natural
reducibility to consider is therefore strong computable reducibility.
In this section, we show that~$\ts^1_k \not \leq_{sc} \ts^1_{k+1}$.
We could have proven this separation using the notion of hyperimmunity 
as we did in the previous section (and this is indeed the approach chosen
by the author in~\cite{Patey2015Iterative}). However, we want to apply preservation
of definitions, emulating Wang's analysis of theorems in reverse mathematics in terms
of preservation of definitions.

\subsubsection{Preservation of non-c.e.\ definitions}

The notion of preservation of definitions was introduced by Wang in~\cite{Wang2014Definability}, in the context of a new analysis
of principles in reverse mathematics in terms of their definitional strength.
Wang defined a set~$X$ to preserve properly~$\Delta^0_2$ definitions
if every properly~$\Delta^0_2$ set (i.e.\ $\Delta^0_2$ but neither $\Sigma^0_1$ nor $\Pi^0_1$) is properly~$\Delta^{0,X}_2$. He deduced several
separation results, and in particular constructed an~$\omega$-model of the conjunction
of~$\coh$, $\wkl$, the Erd\H{o}s-Moser theorem ($\emo$), 
the rainbow Ramsey theorem for pairs ($\rrt^2_2$) and the~$\Pi^0_1$-genericity principle ($\pizog$) 
which is not a model of~$\ts^2$.
His analysis has been extended by the author in~\cite{Patey2016Controlling}.

\begin{definition}[Preservation of non-c.e.\ definitions]\ 
\begin{itemize}
	\item[1.]
A set $X$ \emph{preserves non-c.e.\ definitions} 
of some non-c.e.\ sets~$A_0, A_1, \dots$ 
if no $A_i$ is $X$-c.e.

	\item[2.]
A~$\Pi^1_2$ statement~$\Psf$ \emph{admits preservation of $k$ non-c.e.\ definitions} 
if for each $C$, each sequence of non-$C$-c.e.\ sets~$A_0, \dots, A_{k-1}$ and each $C$-computable $\Psf$-instance $X$, 
there exists a solution $Y$ to~$X$ such that $Y \oplus C$ preserves non-c.e.\ definitions of~$A_0, \dots, A_{k-1}$.
\end{itemize}
\end{definition}

Wang proved~\cite{Wang2014Definability} that
$\coh$, $\emo$, $\wkl$, $\rrt^2_2$ and~$\pizog$ admit preservation of~$k$ non-c.e.\ 
definitions for every~$k \in \omega$.
By a trivial adaptation of Proposition~2.4 from~\cite{Wang2014Definability},
if some statement~$\Psf$ admits preservation of $k$ non-c.e.\ definitions
and some other statement~$\Qsf$ does not, then there exists an~$\omega$-model of~$\Psf$
which is not a model of~$\Qsf$. We start with a trivial lemma showing that our preservation
proofs subsume Wang's analysis of cone avoidance of Ramsey-type theorems from~\cite{Wang2014Some}.

\begin{lemma}
If some statement~$\Psf$ admits preservation of 1 non-c.e.\ definition,
then it admits cone avoidance.
\end{lemma}
\begin{proof}
Fix any set~$C$, any set $A \not \leq_T C$ and any $C$-computable~$\Psf$-instance~$X$.
As~$A \not \leq_T C$, either $A$ or~$\overline{A}$ is not $C$-c.e.
Call this set~$B$. As $\Psf$ admits preservation of 1 non-c.e.\ definition,
there exists a solution $Y$ of $X$
such that~$B$ is not $Y \oplus C$-c.e. 
In particular~$A$ is not~$Y \oplus C$-computable.
\end{proof}

\subsubsection{Negative preservation results}

The following theorem can be proven by a direct adaptation
of Theorem~4.3 proven by Wang~\cite{Wang2014Definability}. However, we provide a simpler proof.

\begin{theorem}\label{thm:ts1-non-preservation}
For every~$k \geq 2$, there exists a~$\Delta^0_2$ $k$-partition
$B_0 \cup \dots \cup B_{k-1} = \omega$ such that for each~$j < k$,
$\overline{B_j}$ is non-c.e.\ but is $H$-c.e.\ for every infinite set $H \subseteq \overline{B_j}$.
\end{theorem}
\begin{proof}
It suffices to construct a stable computable function~$f : [\omega]^2 \to k$
with no infinite computable $f$-thin set, and
such that for each~$i < k$ and each~$x < y < z \in \omega$,
\[
f(x,y) \neq i \wedge f(y,z) \neq i \imp f(x,z) \neq i
\]

We first justify that those properties are sufficient for proving our theorem.
Let~$B_i = \{x : \lim_s f(x,s) = i \}$.
Every infinite subset of~$\overline{B_i}$ computes an infinite set thin for~$f$ with witness~$i$,
therefore no $\overline{B_i}$ is c.e.
Moreover, $\overline{B_i}$ is $H$-c.e.\ for every infinite set $H \subseteq \overline{B_i}$ since
\[
\overline{B_i} = \{ x \in \omega : (\exists y \in H)f(x, y) \neq i \}
\]
The construction of the function~$f$ is done by a finite injury priority argument
with a movable marker procedure. We want to satisfy the following scheme of requirements for each~$e \in \omega$ and~$i < k$:
\[
\Rcal_{e,i} : W_e \mbox{ infinite } \imp (\exists x \in W_e)\lim_s f(x,s) = i
\]

The requirements are given the usual priority ordering.
We proceed by stages, maintaining~$k$ sets~$B_0, \dots, B_{k-1}$
which represent the limit of the function~$f$.
At stage 0, $B_{i,0} = \emptyset$ for each~$i < k$
and $f$ is nowhere defined. Moreover, each requirement~$\Rcal_{e,i}$
is given a movable marker~$m_{e,i}$ initialized to~0.

A strategy for~$\Rcal_{e,i}$ \emph{requires attention at stage~$s+1$}
if $W_{e,s} \subset \overline{B_{i,s}}$ and $W_{e,s} \cap [m_{e,i},s] \neq \emptyset$.
The strategy sets~$B_{i,s+1} = B_{i,s} \cup [m_{e,i},s]$, and~$B_{j,s+1} = B_{j,s} \setminus [m_{e,i},s]$ for every~$j \neq i$.
Then it is declared~\emph{satisfied} until some strategy of higher priority changes its marker. 
Each marker~$m_{e',i'}$ of strategies of lower priorities is assigned the value~$s+1$.

At stage~$s+1$, assume that~$B_{0,s} \cup \dots \cup B_{k-1,s} = [0,s)$
and that~$f$ is defined for each pair over~$[0,s)$.
For each~$x \in [0,s)$, set~$f(x,s) = i$ for the unique~$i$ such that~$x \in B_{i,s}$.
If some strategy requires attention at stage~$s+1$, take the least one
and satisfy it.
If no such requirement is found, set~$B_{0,s+1} = B_{0,s} \cup \{s\}$
and~$B_{i,s+1} = B_{i,s}$ for $i > 0$.
Then go to the next stage. This ends the construction.

Each time a strategy acts, it changes the markers of strategies of lower priority, and is declared satisfied.
Once a strategy is satisfied, only a strategy of higher priority can injury it.
Therefore, each strategy acts finitely often and the markers stabilize.
It follows that the $B$'s also stabilize and that~$f$ is a stable function.

\begin{lemma}
For every~$i < k$ and every~$x < y < z$, $f(x,y) \neq i \wedge f(y,z) \neq i \imp f(x,z) \neq i$.
\end{lemma}
\begin{proof}
Suppose that $f(x,y) \neq i$ but $f(x,z) = i$
for some~$i < k$. Let~$s \leq z$ be the least stage such that~$f(x, t) = i$
for every~$t \in [s+1, z]$. At stage~$s+1$, some strategy~$\Rcal_{e,i}$
moved to~$B_i$ the whole interval~$[m_{e,i},s]$. Since $m_{e',i'} \leq m_{e,i}$
for every strategy~$\Rcal_{e',i'}$ of higher priority, none of the elements in~$[m_{e,i}, s]$ leave~$B_i$ before stage~$z+1$.
As $f(x,y) \neq i$, $y \not \in [s+1,z]$ so $y \in [m_{e,i},s]$.
Therefore $y \in B_{i,z}$ and thus $f(y, z) = i$.
\end{proof}

\begin{lemma}
For every~$e \in \omega$ and~$i < k$, $\Rcal_{e,i}$ is satisfied.
\end{lemma}
\begin{proof}
By induction over the priority order. Let~$s_0$ be a stage after which
no strategy of higher priority will ever act. By construction, $m_{e,i}$ will not change after stage~$s_0$.
If $W_e$ is infinite, it will eventually enumerate some element $u$ bigger than~$m_{e,i}$,
and therefore $\Rcal_{e,i}$ will require attention at some stage~$s \geq u$.
As no strategy of higher priority ever acts after stage~$s_0$, $\Rcal_{e,i}$
will receive attention, be satisfied and never be injured.
\end{proof}

Satisfying $\Rcal_{e,i}$ for every~$e \in \omega$ and $i < k$ guarantees that $f$
has no computable thin set. This last claim finishes the proof of Theorem~\ref{thm:ts1-non-preservation}.
\end{proof}

\begin{corollary}\label{cor:ts2-non-preservation}
For every~$k \geq 2$, $\sts^2_k$ does not admit preservation of~$k$ non-c.e.\ definitions.
\end{corollary}

\subsubsection{Strong preservation of non-c.e.\ definitions}

Because every computable instance of~$\ts^1_k$ having a computable solution,
$\ts^1_k$ admits preservation of $k$ non-c.e.\ definitions for every~$k$.
On the other hand, we would like to say that~$\ts^1_k$ does not \emph{combinatorially}
preserve $k$ non-c.e.\ definitions since Theorem~\ref{thm:ts1-non-preservation} shows
the existence of a non-effective instance of~$\ts^1_k$ whose solutions do not preserve $k$ non-c.e.\ definitions.
This combinatorial notion of preservation is called \emph{strong preservation}.

\begin{definition}[Strong preservation of non-c.e.\ definitions]\ 
A~$\Pi^1_2$ statement~$\Psf$ \emph{admits strong preservation of $k$ non-c.e.\ definitions} 
if for each set $C$, each sequence of non-$C$-c.e.\ sets $A_0, \dots, A_{k-1}$ and each (arbitrary) $\Psf$-instance $X$, 
there exists a solution $Y$ to~$X$ such that $Y \oplus C$ preserves non-c.e.\ definitions of~$A_0, \dots, A_{k-1}$.
\end{definition}

We have seen through Theorem~\ref{thm:ts1-non-preservation} that for every~$k \geq 2$,
$\ts^1_k$ does not admit strong preservation of~$k$ non-c.e.\ definitions.
The following theorem shows the optimality of Theorem~\ref{thm:ts1-non-preservation}.

\begin{theorem}\label{thm:ts1-strong-preservation}
For every~$k \geq 2$, $\ts^1_{k+1}$ admits strong preservation of~$k$ non-c.e.\ definitions.
\end{theorem}

The proof of Theorem~\ref{thm:ts1-strong-preservation} follows Corollary~\ref{cor:sts2k-not-leqc-sts2l}.
Putting Theorem~\ref{thm:ts1-non-preservation} and Theorem~\ref{thm:ts1-strong-preservation}
together, we obtain the desired separation over strong computable reducibility.

\begin{corollary}
For every~$\ell > k \geq 2$, $\ts^1_k \not \leq_{sc} \ts^1_\ell$
\end{corollary}

Using the computable equivalence between the problem of finding a infinite set thin
for an~$\Delta^0_2$ $\ell$-partition and~$\sts^2_\ell$, we deduce the following corollary.

\begin{corollary}\label{cor:sts2k-not-leqc-sts2l}
For every~$\ell > k \geq 2$, $\sts^2_k \not \leq_{c} \sts^2_\ell$
\end{corollary}
\begin{proof}
Fix~$\ell > k \geq 2$ and consider the~$\Delta^0_2$ $k$-partition~$B_0 \cup \dots \cup B_{k-1} = \omega$
of Theorem~\ref{thm:ts1-non-preservation}.
By Schoenfield's limit lemma~\cite{Shoenfield1959degrees}, there exists a stable computable function~$g : [\omega]^2 \to k$
such that~$B_j = \{ x : \lim_s g(x,s) = j \}$ for each~$j < k$. Every infinite set thin for~$g$
is thin for the~$B$'s. Fix any stable computable function~$f : [\omega]^2 \to \ell$
and let~$A_i = \{x : \lim_s f(x,s) = i\}$ for each~$i < m$. By Theorem~\ref{thm:ts1-strong-preservation},
there exists an infinite set~$H$ thin for the~$A$'s which does not compute an infinite set thin for the~$B$'s
(hence for~$g$). As~$H \oplus f$ computes an infinite set~$G$ thin for~$f$, $f$ has an infinite $f$-thin set
which does not compute an infinite set thin for~$g$.
\end{proof}

The remainder of this section is devoted to the proof of Theorem~\ref{thm:ts1-strong-preservation}.
Fix some set~$C$ preserving non-c.e.\ definitions of some sets~$B_0, \dots, B_{k-1}$
and fix some $(k+1)$-partition~$A_0 \cup \dots \cup A_k = \omega$.
We will construct a set~$G$ such that $G \cap \overline{A_i}$
is infinite for each~$i \leq k$ and none of the $B$'s are $(G \cap \overline{A_i}) \oplus C$-c.e.
for some~$i \leq k$.
Our forcing conditions are Mathias conditions~$(\sigma, X)$ where
$X$ is an infinite set of integers such that none of the $B$'s are $X \oplus C$-c.e.

\subsubsection{Forcing limitlessness}

We want to satisfy the following scheme of requirements to ensure that~$G \cap \overline{A_i}$
is infinite for each~$i \leq k$:

\[
\Qcal_p : (\exists m_0, \dots, m_k > p)[m_0 \in G \cap \overline{A_0} \wedge \dots \wedge m_k \in G \cap \overline{A_k}]
\]

We say that an~$(k+1)$-partition~$A_0 \cup \dots \cup A_k = \omega$ is \emph{non-trivial} 
if there exists no infinite set~$H$ homogeneous for the~$A$'s
such that none of the $B$'s are $H \oplus C$-c.e. Of course, every infinite set homogeneous for the~$A$'s
is thin for the~$A$'s, so if the partition~$A_0 \cup \dots \cup A_k = \omega$ is trivial, we succeed.
Therefore we will assume from now on that the partition is non-trivial.
A condition~$c = (\sigma, X)$ \emph{forces $\Qcal_p$}
if there exist some~$m_0, \dots, m_k > p$ such that~$m_i \in \sigma \cap \overline{A_i}$
for each~$i \leq k$. Therefore, if~$G$ satisfies~$c$ and~$c$ forces~$\Qcal_p$,
then~$G$ satisfies the requirement~$\Qcal_p$.
We now prove that the set of conditions forcing~$\Qcal_p$ is dense for each~$p \in \omega$.
Thus, every sufficiently generic filter will induce a set~$G$
such that $G \cap \overline{A_i}$ is infinite for each~$i \leq k$.

\begin{lemma}\label{lem:ts2-reduc-force-Q}
For every condition~$c$ and every~$p$, there is an extension forcing~$\Qcal_p$.
\end{lemma}
\begin{proof}
Fix some~$p \in \omega$. It is sufficient to show that given a condition~
$c = (\sigma, X)$ and some~$i \leq k$,
there exist an extension~$d_0 = (\tau, Y)$ and some integer~$m_i > p$
that~$m_i \in \tau \cap \overline{A_i}$.
By iterating the process for each~$i \leq k$, we obtain an extension forcing~$\Qcal_p$.
Suppose for the sake of contradiction that~$X \cap \overline{A_i}$ is finite.
One can then $X$-compute an infinite set~$H \subseteq A_i$, 
contradicting non-triviality of the~$A$'s.
Therefore, there exists an~$m_i \in X \cap \overline{A_i}$ such that $m_i > max(\sigma, p)$.
The condition $d_0 = (\sigma^{\concat} m_i, X)$
is the desired extension.
\end{proof}

\subsubsection{Forcing preservation}

The second scheme of requirements consists in ensuring that
the sets~$B_0, \dots, B_{k-1}$ are all non-$(G \cap \overline{A_i}) \oplus C$-c.e.\ for some~$i \leq k$.
The requirements are of the following form for each tuple of indices~$\vec{e} = (e_i : i \leq k)$:

\[
\Rcal_{\vec{e}} : \bigwedge_{j < k} W^{(G \cap \overline{A_0}) \oplus C}_{e_0} \neq B_j
	\vee \dots \vee \bigwedge_{j < k} W^{(G \cap \overline{A_k}) \oplus C}_{e_k} \neq B_j
\]

A condition~\emph{forces $\Rcal_{\vec{e}}$} if every set~$G$ satisfying this condition also satisfies requirement~$\Rcal_{\vec{e}}$.
The following lemma is the core of the forcing argument.

\begin{lemma}\label{lem:ts2-reduc-step}
For every condition~$c = (\sigma, X)$, every~$i_0 < i_1 \leq k$,
every $j < k$ and every vector of indices~$\vec{e}$, there exists
an extension~$d$ forcing either~$W^{(G \cap \overline{A_{i_0}}) \oplus C}_{e_{i_0}} \neq B_j$
or $W^{(G \cap \overline{A_{i_1}}) \oplus C}_{e_{i_1}} \neq B_j$.
\end{lemma}
\begin{proof}
Let~$W$ be the set of all $a \in \omega$ such that for every $2$-cover~$Z_{i_0} \cup Z_{i_1} = X$, 
there is some $i \in \{i_0,i_1\}$ and some set~$G_i \subseteq Z_i$ 
such that~$a \in W_{e_i}^{(G_i/(\sigma \cap \overline{A_i})) \oplus C}$.
The set~$W$ is~$X \oplus C$-c.e. Therefore~$W \neq B_j$. Let~$a \in W \Delta B_j$. We have two cases:
\begin{itemize}
	\item Case 1: $a \in W \setminus B_j$.
	By definition of~$W$, taking in particular the sets $Z_{i_0} = X \cap \overline{A_{i_0}}$
	and~$Z_{i_1} = X \cap \overline{A_{i_1}}$, there is some~$i \in \{i_0,i_1\}$
	and some finite set~$G_i \subseteq Z_i$ such that~$a \in W_{e_i}^{(G_i/(\sigma \cap \overline{A_i}))\oplus C}$.
	The condition~$d = (G_i/\sigma, X)$
	is an extension forcing~$W_{e_i}^{(G \cap \overline{A_i}) \oplus C} \neq B_j$.

	\item Case 2: $a \in B_j \setminus W$.
	Let $\Ccal$ be the $\Pi^{0,X \oplus C}_1$ class of sets $Z_{i_0} \oplus Z_{i_1}$ such that
	$Z_{i_0} \cup Z_{i_1} = X$ and for every~$i \in \{i_0,i_1\}$
	and every set~$G_i \subseteq Z_i$ $a \not \in W_{e_i}^{(G_i/(\sigma \cap \overline{A_i})) \oplus C}$.
	By definition of~$W$, $\Ccal$ is non-empty. 
	As $\wkl$ admits preservation of $k$ non-c.e.\ definitions, there exists some~$Z_{i_0} \oplus Z_{i_1} \in \Ccal$
	such that none of the~$B$'s are~$Z_{i_0} \oplus Z_{i_1} \oplus X \oplus C$-c.e. 
	Let~$i \in \{i_0,i_1\}$ be such that $Z_i$ is infinite. The condition~$d = (\sigma, Z_i)$
	is an extension of~$c$ forcing~$W_{e_i}^{(G \cap \overline{A_i}) \oplus C} \neq B_j$.
\end{itemize}
\end{proof}

As usual, the following lemma iterates Lemma~\ref{lem:ts2-reduc-step}
and uses the fact that~$k+1 > k$ to satisfy the requirement~$\Rcal_{\vec{e}}$.

\begin{lemma}\label{lem:ts2-reduc-force-R}
For every condition~$c$,
and every indices~$\vec{e}$, there exists
an extension~$d$ forcing~$\Rcal_{\vec{e}}$.
\end{lemma}
\begin{proof}
Fix a condition~$c$, and iterate applications of Lemma~\ref{lem:ts2-reduc-step}
to obtain an extension~$d$ such that for each~$j < k$, 
$d$ forces~$W_{e_i}^{(G \cap \overline{A_i}) \oplus C} \neq B_j$ for $k$ different~$i$'s.
By the pigeonhole principle, there exists some~$i \leq k$
such that $d$ forces~$W_{e_i}^{(G \cap \overline{A_i}) \oplus C} \neq B_j$ for each~$j < k$.
Therefore, $d$ forces~$\Rcal_{\vec{e}}$.
\end{proof}

\subsubsection{Construction}

Thanks to Lemma~\ref{lem:ts2-reduc-force-Q} and Lemma~\ref{lem:ts2-reduc-force-R}, 
define an infinite descending sequence of conditions
$(\varepsilon, \omega) \geq c_0 \geq \dots$ such that for each~$s \in \omega$,
\begin{itemize}
	\item[(a)] $c_s$ forces~$\Qcal_s$
	\item[(b)] $c_s$ forces~$\Rcal_{\vec{e}}$ if~$s = \tuple{\vec{e}}$
\end{itemize}
where~$c_s = (\sigma_s, X_s)$. Let~$G = \bigcup_s \sigma_s$.
By (a), $G \cap \overline{A_i}$ is infinite for every~$i \leq k$ and by (b),
$G$ satisfies each requirement~$\Rcal_{\vec{e}}$.
This finishes the proof of Theorem~\ref{thm:ts1-strong-preservation}.

\subsection{Thin set theorem for pairs and reverse mathematics}\label{sect:preservation-ts2-omega-models}

There is a fundamental difference in the way we proved that~$\rt^1_k \not \leq_{sc} \rt^1_\ell$
and that~$\ts^1_\ell \not \leq_{sc} \ts^1_k$ whenever~$k > \ell$. In the former case,
we have built an instance~$I$ of~$\rt^1_k$ satisfying some hyperimmunity properties,
and used those properties to construct a solution~$X$ to each instance of~$\rt^1_\ell$
which does not compute a solution to~$I$. We did not ensure that those hyperimmunity properties
are preserved relative to the solution~$X$, which prevents us from iterating the construction.
As it happens, those properties are not preserved as multiple applications
of~$\rt^1_\ell$ are sufficient to compute a solution to~$I$.
In the latter case, we proved that~$\ts^1_\ell$ has an instance whose solutions do not preserve some definitional property,
whereas each instance of~$\ts^1_k$ has a solution preserving it. This preservation enables us to
iterate the applications of~$\ts^1_k$ and build $\omega$-structures whose
second-order part is made of sets preserving this property.
We will take advantage of those observations to obtain new separations in reverse mathematics.

In this section, we prove that~$\ts^2_{k+1}$ does not imply~$\ts^2_k$ over~$\rca$
for every~$k \geq 2$. In particular, we answer several questions
asked by Cholak, Giusto, Hirst and Jockusch~\cite{Cholak2001Free} and by Mont\'alban~\cite{Montalban2011Open}
about the relation between~$\rt^2_2$ and~$\ts^2$.
Dorais et al.~\cite{Dorais2016uniform} proved that~$\rca \vdash \ts^n_k \imp \aca$ for~$n \geq 3$ 
whenever~$k$ is not large enough. Therefore we cannot hope to obtain the same separation
result over~$\rca$ for arbitrary tuples. However, we shall see that
$\ts^n_k$ is not computably reducible to~$\ts^n_{k+1}$ for~$n, k \geq 2$.

\begin{theorem}\label{thm:ts2-omega-models}
For every~$k \geq 2$,
let~$\Phi$ be the conjunction of~$\coh$, $\wkl$, $\rrt^2_2$,
$\pizog$, $\emo$, $\ts^2_{k+1}$.
Over~$\rca$, $\Phi$ implies neither~$\sts^2_k$ nor~$\sads$.
\end{theorem}

The proof of Theorem~\ref{thm:ts2-omega-models} follows Corollary~\ref{cor:ts2-preservation}.
Cholak et al.~\cite{Cholak2001Free} and Mont\'alban~\cite{Montalban2011Open} asked whether~$\ts^2$ implies~$\rt^2_2$
over~$\rca$. Thanks to Theorem~\ref{thm:ts2-omega-models}, we answer negatively,
noticing that~$\ts^2_2$ is the statement~$\rt^2_2$
and~$\rca \vdash \ts^2_k \imp \ts^2$ for every~$k \geq 2$ (see Dorais et al.~\cite{Dorais2016uniform}).

\begin{corollary}
$\ts^2$ does not imply~$\rt^2_2$ over~$\rca$.
\end{corollary}

Using the standard trick of prehomogeneous sets, we can generalize
from computable non-reducibility over pairs to arbitrary tuples.

\begin{corollary}\label{cor:ts-non-computable-reduction}
For every~$k, n \geq 2$ there exists a~$\Delta^0_n$ $k$-partition
$B_0 \cup \dots \cup B_{k-1} = \omega$ such that every computable
coloring $f : [\omega]^n \to k+1$ has an infinite $f$-thin set
computing no set thin for the~$B$'s. 
\end{corollary}
\begin{proof}
This is proved in a relativized form by induction over~$n \geq 2$.
The case~$n = 2$ is obtained by relativizing the proof of Theorem~\ref{thm:ts2-omega-models},
which shows indeed the existence of a $\Delta^0_2$ $k$-partition $B_0 \cup \dots \cup B_{k-1} = \omega$
such that every computable coloring $f : [\omega]^2 \to k+1$ has an infinite $f$-thin set computing no set thin for the~$B$'s.
Now assume it holds for some~$n$ in order to prove it for~$n+1$.
By the relativized low basis theorem~\cite{Jockusch197201},
let~$P \gg \emptyset^{(n-1)}$ be such that~$P' \leq \emptyset^{(n)}$.
Applying the induction hypothesis to~$P$,
there is a~$\Delta^{0,P}_2$ (hence~$\Delta^0_{n+1}$) $k$-partition~$B_0 \cup \dots \cup B_{k-1} = \omega$
such that each $P$-computable coloring $f: [\omega]^n \to k+1$
has an infinite $f$-homogeneous set~$H$ such that~$H \oplus P$
does not compute an infinite set thin for the~$B$'s.

Let~$f : [\omega]^{n+1} \to k+1$ be a computable coloring.
By Jockusch~\cite[Lemma 5.4]{Jockusch1972Ramseys}, there exists an infinite set~$C$ pre-homogeneous for~$f$
such that~$C \leq_T P$.
Let~$\tilde{f} : [C]^n \to k+1$ be the~$P$-computable coloring defined for each~$\sigma \in [C]^n$
by~$\tilde{f}(\sigma) = f(\sigma, a)$, where~$a \in A$, $a > max(\sigma)$.
Every $\tilde{f}$-thin set is~$f$-thin.
By definition of~$B_0 \cup \dots \cup B_{k-1} = \omega$, 
there exists an infinite~$\tilde{f}$-thin (hence $f$-thin) set~$H$
such that $H \oplus P$ does not compute an infinite set thin for the~$B$'s.
\end{proof}

\begin{corollary}
For every~$k, n \geq 2$, $\sts^n_k \not \leq_c \ts^n_{k+1}$
\end{corollary}

We proved in section~\ref{sect:ts-strong-reduc} that $\sts^2_k$ does not admit
preservation of~$k$ non-c.e.\ definitions. Jockusch noticed (see Hirschfeldt and Shore~\cite{Hirschfeldt2007Combinatorial})
that $\sads$ does not admit preservation of 2 non-c.e.\ definitions. 
We give the proof for the sake of completeness.

\begin{theorem}\label{thm:sads-non-preservation}
$\sads$ does not admit preservation of~2 non-c.e.\ definitions.
\end{theorem}
\begin{proof}
Tennenbaum (see Rosenstein~\cite{Rosenstein1982Linear})
constructed a computable linear order of order type~$\omega+\omega^{*}$
with no computable infinite ascending or descending sequence.
Let~$B_0$ be the~$\omega$-part and~$B_1$ be the $\omega^{*}$ part of this linear order.
Every infinite subset of~$B_0$ (resp. $B_1$) computes an infinite ascending (resp. descending) sequence,
therefore~$B_0$ and~$B_1$ are non-c.e.
The~$\omega$ part (resp. $\omega^{*}$ part) is c.e.\ in every infinite ascending (resp. descending) sequence.
\end{proof}

By Schoenfield's limit lemma~\cite{Shoenfield1959degrees}, a stable computable coloring over $(n+1)$-tuples
can be considered as a non-effective coloring over $n$-tuples. This consideration establishes
a bridge between preservation properties for colorings over~$(n+1)$-tuples and strong preservation 
properties for colorings over $n$-tuples. In particular, it enables us
to prove preservation results by induction over~$n$.
The following lemma has been proven by the author in its full generality in~\cite{PateyCombinatorial}.
Nevertheless we reprove it in the context of preservation of non-c.e.\ definitions.

\begin{lemma}\label{lem:ts-strong-to-weak}
For every~$k,n \geq 1$ and~$\ell \geq 2$, if~$\ts^n_\ell$ admits strong preservation of~$k$ non-c.e.\ definitions,
then $\ts^{n+1}_\ell$ admits preservation of~$k$ non-c.e.\ definitions.
\end{lemma}
\begin{proof}
Fix any set~$C$, $k$ non-$C$-c.e.\ sets~$A_0, \dots, A_{k-1}$ and any $C$-computable
coloring $f : [\omega]^{n+1} \to \ell$.
Consider the uniformly~$C$-computable sequence of sets~$\vec{R}$ defined for each~$\sigma \in [\omega]^n$ and~$i < \ell$ by
\[
R_{\sigma,i} = \{s \in \omega : f(\sigma,s) = i\}
\]
As~$\coh$ admits preservation of~$k$ non-c.e.\ definitions, there exists
some~$\vec{R}$-cohesive set~$G$ such that $G \oplus C$ preserves non-c.e.\ definitions
of the~$A$'s. The cohesive set induces a $(G \oplus C)'$-computable coloring~$\tilde{f} : [\omega]^n \to \ell$ defined by:
\[
(\forall \sigma \in [\omega]^n) \tilde{f}(\sigma) = \lim_{s \in G} f(\sigma,s)
\]
As ~$\ts^n_\ell$ admits strong preservation of $k$ non-c.e.\ definitions,
there exists an infinite $\tilde{f}$-thin set~$H$ such that
$H \oplus G \oplus C$ preserves non-c.e.\ definitions
of the~$A$'s. $H \oplus G \oplus C$ computes an infinite $f$-thin set.
\end{proof}

Using Theorem~\ref{thm:ts1-strong-preservation} together with Lemma~\ref{lem:ts-strong-to-weak}, 
we deduce the following corollary.

\begin{corollary}\label{cor:ts2-preservation}
For every~$k \geq 2$, $\ts^2_{k+1}$ admits preservation of~$k$ non-c.e.\ definitions.
\end{corollary}

We are now ready to prove Theorem~\ref{thm:ts2-omega-models}.

\begin{proof}[Proof of Theorem~\ref{thm:ts2-omega-models}]
Fix some~$k \geq 2$.
Wang proved in~\cite{Wang2014Definability} that $\coh$, $\wkl$, $\rrt^2_2$,
$\pizog$ and $\emo$ admit preservation of $k$ non-c.e.\ definitions.
By Corollary~\ref{cor:ts2-preservation}, so does $\ts^2_{k+1}$.
By Corollary~\ref{cor:ts2-non-preservation} and Theorem~\ref{thm:sads-non-preservation},
neither~$\sts^2_k$ nor~$\sads$ admit preservation of~$k$ non-c.e.\ definitions.
The theorem follows by an application of Proposition 2.4 of Wang~\cite{Wang2014Definability}.
\end{proof}

\subsection{Thin set theorem for tuples and reverse mathematics}\label{sect:preservation-ts-omega-models}

In this section, we extend the preservation of non-c.e.\ definitions of the thin set theorem for pairs
to arbitrary tuples, using the same construction pattern as Wang~\cite{Wang2014Some}.
We deduce that $\ts^n_\ell$ does not imply~$\ts^n_k$ over~$\rca$ whenever~$\ell$ is large enough,
which is informally the strongest result we can obtain since Proposition~5.3 in Dorais et al.~\cite{Dorais2016uniform}
states that $\rca \vdash \aca \biimp \ts^n_k$ for~$n \geq 3$ whenever~$k$ is not large enough.

\begin{theorem}\label{thm:ts-arbitrary-preservation}
For every~$k,n \geq 1$, $\ts^n_\ell$ admits strong preservation of $k$ non-c.e.\ definitions 
for sufficiently large~$\ell$. 
\end{theorem}

The proof of Theorem~\ref{thm:ts-arbitrary-preservation} begins below in section~\ref{subsect:ts-proof-structure}.
Using the fact that~$\rca \vdash \ts^n_\ell \imp \ts^n$ for every~$n, \ell \geq 2$,
we obtain the following preservation result for~$\ts$.

\begin{corollary}
For every~$k \geq 1$, $\ts$ admits strong preservation of $k$ non-c.e.\ definitions.
\end{corollary}

Thanks to the existing preservations of non-c.e.\ definitions
and Proposition~2.4 from Wang~\cite{Wang2014Definability}, we deduce the following separations
over~$\omega$-models.

\begin{corollary}
For every~$k \geq 2$,
let~$\Phi$ be the conjunction of~$\coh$, $\wkl$, $\rrt^2_2$,
$\pizog$, $\emo$, $\ts^2_{k+1}$ and~$\ts$.
Over~$\rca$, $\Phi$ implies neither~$\sts^2_k$ nor~$\sads$.
\end{corollary}

The remainder of this section is devoted to the proof of Theorem~\ref{thm:ts-arbitrary-preservation}.

\subsubsection{Proof structure}\label{subsect:ts-proof-structure}

We shall follow the proof structure
of strong cone avoidance used by Wang~\cite{Wang2014Some}. Fix some~$k \geq 1$.
The induction works as follows:

\begin{itemize}
	\item[(A1)] In section~\ref{sect:preservation-ts2-omega-models}
	we proved that~$\ts^1_{k+1}$ admits strong preservation of $k$ non-c.e.\ definitions.
	This is the base case of our induction.
	\item[(A2)] Assuming that for each~$t \in (0,n)$, $\ts^t_{d_t+1}$ admits
	strong preservation of $k$ non-c.e.\ definitions, we prove
	that~$\ts^n_{d_{n-1}+1}$ admits preservation of $k$ non-c.e.\ definitions.
	This is done by Lemma~\ref{lem:ts-strong-to-weak}.
	\item[(A3)] Then we prove that~$\ts^n_{d_n+1}$ admits strong preservation of
	$k$ non-c.e.\ definitions where
	\[
	d_n = d_1 d_{n-1} + \sum_{0 < t < n} d_t d_{n-t}
	\]
\end{itemize}

Properties (A1) and (A2) are already proven. We now prove property (A3). It is again done in several steps.
Fix a coloring $f : [\omega]^n \to d_n+1$ and a set $C$ preserving non-c.e.\ definitions of $k$ sets~$A_0, \dots, A_{k-1}$.

\begin{itemize}
	\item[(S1)] First, we construct an infinite set $D \subseteq \omega$ such that $D \oplus C$ preserves non-c.e.\ definitions of the~$A$'s
	and a sequence $(I_\sigma : 0 < |\sigma| < n)$ such that for each $t \in (0,n)$ and each $\sigma \in [\omega]^t$
	\begin{itemize}
		\item[(a)] $I_\sigma$ is a subset of $\{0, \dots, d_n\}$ with at most $d_{n-t}$ many elements  
		\item[(b)] $(\exists b)(\forall \tau \in [D \cap (b,+\infty)]^{n-t}) f(\sigma,\tau) \in I_\sigma$
	\end{itemize}

	\item[(S2)]
	Second, we construct an infinite set~$E \subseteq D$ such that $E \oplus C$ preserves non-c.e.\ definitions of the~$A$'s
	and a sequence $(I_t : 0 < t < n)$ such that for each $t \in (0,n)$
	\begin{itemize}
		\item[(a)] $I_t$ is a subset of~$\{0,\dots, d_n\}$ of size at most $d_t d_{n-t}$
		\item[(b)] $(\forall \sigma \in [E]^t)(\exists b)(\forall \tau \in [E \cap (b,+\infty)]^{n-t}) f(\sigma,\tau) \in I_t$
	\end{itemize}

	\item[(S3)] Third, we construct a sequence~$(\xi_i \in [E]^{<\omega} : i < \omega)$ such that
	\begin{itemize}
		\item[(a)] The set $G = \bigcup_i \xi_i$ is infinite and $G \oplus C$ preserves non-c.e.\ definitions of the~$A$'s
		\item[(b)] $|f([\xi_i]^n)| \leq d_{n-1}$ and~$max(\xi_i) < min(\xi_{i+1})$ for each~$i < \omega$
		\item[(c)] For each~$t \in (0,n)$ and~$\sigma \in [\bigcup_{j < i} \xi_j]^t$,
		$f(\sigma,\tau) \in I_t$ for all~$\tau \in [\bigcup_{j \geq i} \xi_j]^{n-t}$
	\end{itemize}

	\item[(S4)] Finally, we build an infinite set~$H \subseteq G$
	such that~$H \oplus C$ preserves non-c.e.\ definitions of the~$A$'s
	and~$|f([H]^n)| \leq d_n$.
\end{itemize}

\subsubsection{Generalized cohesiveness}

Before proving that~$\ts^n_{d_n+1}$ admits strong preservation of~$k$ non-c.e.\ definitions,
we need to prove strong preservation for a generalized notion of cohesiveness
already used by the author in~\cite{PateyCombinatorial}.
Cohesiveness can be seen as the problem which takes as an input a coloring of pairs $f : [\omega]^2 \to \ell$
and fixes the first parameter to obtain an infinite sequence of colorings of integers $f_x : \omega \to \ell$ for each $x \in \omega$.
A solution to this problem is an infinite set~$G$ which is eventually homogeneous for each coloring $f_x$.

Going further in this approach, we can consider that cohesiveness is a degenerate
case of the problem which takes as an input a coloring of pairs $f : [\omega]^2 \to \omega$ using infinitely many colors,
and fixes again the first parameter to obtain an infinite sequence of colorings of integers $f_x : \omega \to \omega$.
A solution to this problem is an infinite set~$G$ such that for each color $i$, either eventually the color will be avoided
by $f_x$ over $G$, or $G$ will be eventually homogeneous for~$f_x$ with color $i$.

We can generalize the notion to colorings over tuples $f : [\omega]^n \to \omega$,
seeing $f$ as an infinite sequence of colorings over $t$-uples $f_{\sigma} : [\omega]^t \to \omega$ for each $\sigma \in [\omega]^{n - t}$. We will create a set~$G$ such that at most $d_t$ colors will appear for arbitrarily large pairs over~$G$ for each function $f_{\sigma}$.
This set will be constructed by applying $\ts^t_{d_t+1}$ to $f_{\sigma}$ for each~$\sigma$.

We do not need Theorem~\ref{thm:generalized-cohesivity-strong-avoidance} in its 
full generality to complete our step (S1). However, it
will be useful in a later section for proving that the free set theorem admits
preservation of $k$ non-c.e.\ definitions.

\begin{theorem}\label{thm:generalized-cohesivity-strong-avoidance}
Fix a coloring $f : [\omega]^n \to \omega$, some $t \leq n$ and
suppose that
$\ts^s_{d_s+1}$ admits strong preservation of~$k$ non-c.e.\ definitions for each $s \in (0, t]$. 
For every set $C$ preserving non-c.e.\ definitions of some sets~$A_0, \dots, A_{k-1}$,
there exists an infinite set $G$ such that $G \oplus C$ preserves non-c.e.\ definitions of the~$A$'s
and for every $\sigma \in [\omega]^{<\omega}$ such that $n-t \leq \card{\sigma} < n$, 
$$
\card{\set{x : (\forall b)(\exists \tau \in [G \cap (b, +\infty)]^{n-|\sigma|}) f(\sigma, \tau) = x}} \leq d_{n-|\sigma|}
$$
\end{theorem}
\begin{proof}
Our forcing conditions are Mathias conditions $(F, X)$ where $X \oplus C$ preserves non-c.e.\ definitions
of the~$A$'s.
Lemma~3.16 in Wang~\cite{Wang2014Definability} states that for every set~$G$ which is sufficiently
generic for~$(F,X)$, $G \oplus C$ preserves $k$ non-c.e.\ definitions. It suffices therefore to prove the following lemma.

\begin{lemma}\label{lem:gen-coh-function}
For every condition $(F, X)$ and $\sigma \in [\omega]^{<\omega}$ such that $n-t \leq \card{\sigma} < n$,
for every finite set $I$ such that $\card{I} = d_{n-|\sigma|}$, 
there exists an extension $(F, \tilde{X})$ such that 
$$
\{f(\sigma, \tau) : \tau \in [\tilde{X}]^{n-|\sigma|}\} \subseteq I
\hspace{10pt} \mbox{ or } \hspace{10pt} 
I \not \subseteq \{f(\sigma, \tau) : \tau \in [\tilde{X}]^{n-|\sigma|}\}
$$
\end{lemma}
\begin{proof}
Define the function $g : [X]^{n - |\sigma|} \to I \cup \{\bot\}$
by $g(\tau) = f(\sigma, \tau)$ if $f(\sigma, \tau) \in I$
and $g(\tau) = \bot$ otherwise. By strong preservation of~$k$ non-c.e.\ definitions of $\ts^{n - |\sigma|}_{d_{n-|\sigma|}+1}$,
there exists an infinite subset $\tilde{X} \subseteq X$ such that $\tilde{X} \oplus C$
preserves non-c.e.\ definitions of the~$A$'s and $\card{\{g(\tau) : \tau \in [\tilde{X}]^{n-|\sigma|}\}} \leq d_{n-|\sigma|}$.
The condition $(F, \tilde{X})$ is the desired extension.
\end{proof}

Using Lemma~3.16 in~\cite{Wang2014Definability}
and Lemma~\ref{lem:gen-coh-function}, one can define an infinite
descending sequence of conditions $(\emptyset, \omega) \geq c_0 \geq c_1 \geq \dots$
such that for each $s \in \omega$
\begin{itemize}
  \item[1.] $c_s = (F_s, X_s)$ with $\card{F_s} \geq s$
  \item[2.] $c_s$ forces $W_e^{G \oplus C} \neq A_i$ if $s = \tuple{e,i}$
  \item[3.] $\{f(\sigma, \tau) : \tau \in [X_s]^{n-|\sigma|}\} \subseteq I$
  or $I \not \subseteq \{f(\sigma, \tau) : \tau \in [X_s]^{n-|\sigma|}\}$ if $s = \tuple{\sigma, I}$
  and $\card{I} = d_{n-|\sigma|}$.
\end{itemize}
The set $G = \bigcup_s F_s$ is an infinite set 
such that $G \oplus C$ preserves non-c.e.\ definitions of the~$A$'s. We claim that $G$ satisfies the desired properties.
Fix a $\sigma \in [\omega]^{<\omega}$ such that $n-t \leq \card{\sigma} < n$. Suppose that there exists $d_{n - |\sigma|} + 1$ elements 
$x_0, \dots, x_{d_{n-|\sigma|}}$ such that $(\forall b)(\exists \tau \in [G \cap (b, +\infty)]^{n-|\sigma|}) f(\sigma, \tau) = x_i$
for each $i \leq d_{n-|\sigma|}$.
Let $I = \{x_0, \dots, x_{d_{n-|\sigma|}-1}\}$.
By step $s = \tuple{\sigma, I}$, $G$ satisfies $(F_s, X_s)$ such that $\{f(\sigma, \tau) : \tau \in [X_s]^{n-|\sigma|}\} \subseteq I$
or $I \not \subseteq \{f(\sigma, \tau) : \tau \in [X_s]^{n-|\sigma|}\}$. In the first case it contradicts
the choice of $x_{d_{n-|\sigma|}}$ and in the second case it contradicts the choice of an element of $I$.
This finishes the proof of Theorem~\ref{thm:generalized-cohesivity-strong-avoidance}.
\end{proof}

\subsubsection{Step (S1) : Construction of the set~$D$}

We start with the construction of an infinite set~$D \subseteq \omega$ such that $D \oplus C$ preserves non-c.e.\ definitions of the~$A$'s
and a sequence $(I_\sigma : 0 < |\sigma| < n)$ such that for each $t \in (0,n)$ and each $\sigma \in [\omega]^t$
\begin{itemize}
	\item[(a)] $I_\sigma$ is a subset of $\{0, \dots, d_n\}$ with at most $d_{n-t}$ many elements  
	\item[(b)] $(\exists b)(\forall \tau \in [G \cap (b,+\infty)]^{n-t}) f(\sigma,\tau) \in I_\sigma$
\end{itemize}

Let $D$ be the set constructed in Theorem~\ref{thm:generalized-cohesivity-strong-avoidance}
for $t = n-1$. For each $\sigma \in [\omega]^{<\omega}$ such that $0 < \card{\sigma} < n$,
let
$$
I_\sigma = \{x \leq d_n :  (\forall b)(\exists \tau \in [G \cap (b, +\infty)]^{n-|\sigma|}) f(\sigma, \tau) = x\}
$$
By choice of $D$, the set $I_\sigma$ has at most $d_{n-|\sigma|}$ many elements.
Moreover, for each $y \leq d_n$ such that $y \not \in I_\sigma$, there exists a bound $b_y$
such that $(\forall \tau \in [D \cap (b_y, +\infty)]^{n-|\sigma|}) f(\sigma, \tau) \neq x$.
So taking $b = max(b_y : y \leq d_n \wedge y \not \in I_\sigma)$, we obtain
$$
(\forall \tau \in [D \cap (b,+\infty)]^{n-|\sigma|}) f(\sigma,\tau) \in I_\sigma
$$

\subsubsection{Step (S2) : Construction of the set~$E$}

We now construct an infinite set~$E \subseteq D$ such that $E \oplus C$ preserves non-c.e.\ definitions of the~$A$'s
and a sequence $(I_t : 0 < t < n)$ such that for each $t \in (0,n)$
\begin{itemize}
	\item[(a)] $I_t$ is a subset of~$\{0,\dots, d_n\}$ of size at most $d_t d_{n-t}$
	\item[(b)] $(\forall \sigma \in [E]^t)(\exists b)(\forall \tau \in [E \cap (b,+\infty)]^{n-t}) f(\sigma,\tau) \in I_t$
\end{itemize}

For each $t \in (0,n)$ and~$\sigma \in [\omega]^t$, let $F_t(\sigma) = I_\sigma$.
Using strong preservation of $k$ non-c.e.\ definitions
of $\ts^t_{d_t+1}$,
we build a finite sequence $D \supseteq E_1 \supseteq \dots \supseteq E_{n-1}$
such that for each $t \in (0,n)$
\begin{itemize}
	\item[1.] $E_t \oplus C$ preserves non-c.e.\ definitions of the~$A$'s
	\item[2.] $|F_t([E_t]^t)| \leq d_t$
\end{itemize}
Let $E = E_{n-1}$ and $I_t = \bigcup F_t([E]^t)$ for each $t \in (0,n)$.
As for each~$\sigma \in [E]^t$, $|F_t(\sigma)| = |I_\sigma| \leq d_{n-t}$, $|I_t| \leq d_t d_{n-t}$, so property (a) holds.
We now check that property (b) is satisfied.
Fix a $\sigma \in [E]^t$. By definition of~$I_t$, $F_t(\sigma) = I_\sigma \subseteq I_t$.
As $E \subseteq D$,
\[
(\exists b)(\forall \tau \in [E \cap (b,+\infty)]^{n-t}) f(\sigma,\tau) \in I_\sigma \subseteq I_t
\]

\subsubsection{Step (S3) : Construction of the set~$G$}

Given the set~$E$ and the sequence of sets of colors~$(I_t : 0 < t < n)$,
we will construct a sequence~$(\xi_i \in [E]^{<\omega} : i < \omega)$ such that
\begin{itemize}
	\item[(a)] The set $G = \bigcup_i \xi_i$ is infinite and $G \oplus C$ preserves non-c.e.\ definitions of the~$A$'s
	\item[(b)] $|f([\xi_i]^n)| \leq d_{n-1}$ and~$max(\xi_i) < min(\xi_{i+1})$ for each~$i < \omega$
	\item[(c)] For each~$t \in (0,n)$ and~$\sigma \in [\bigcup_{j < i} \xi_j]^t$,
	$f(\sigma,\tau) \in I_t$ for all~$\tau \in [\bigcup_{j \geq i} \xi_j]^{n-t}$
\end{itemize}

We construct our set~$G$ by Mathias forcing~$(\sigma, X)$ where
$X$ is an infinite subset of~$E$ such that $X \oplus C$ preserves non-c.e.\ definitions
of the~$A$'s. Using property~(b) of~$E$, we can easily
construct an infinite sequence $(\xi_i \in [E]^{<\omega} : i < \omega)$
satisfying properties~(b) and~(c) of step (S3). The following lemma
shows how to satisfy property~(a).

\begin{lemma}\label{lem:ts-lemma-set-g}
Fix a condition~$(\sigma, X)$, some~$e \in \omega$ and some~$j < k$.
There exists an extension~$(\sigma\xi, Y)$
with~$|f([\xi]^n)| \leq d_{n-1}$, forcing~$W_e^{G \oplus C} \neq A_j$.
\end{lemma}
\begin{proof}
Let~$W$ be the set of~$a \in \omega$ such that 
for every coloring $g : [X]^n \to d_n+1$, there is a set~$\xi \in [X]^{<\omega}$
such that~$|g([\xi]^n)| \leq d_{n-1}$ and~$a \in W_e^{\sigma\xi \oplus C}$.
The set~$W$ is~$X \oplus C$-c.e, therefore~$W \neq A_j$. Let~$a \in W \Delta A_j$.
We have two cases:
\begin{itemize}
	\item Case 1: $a \in W \setminus A_j$. In particular, taking~$g = f$,
	there exists a set~$\xi \in [X]^{<\omega}$
	such that~$|f([\xi]^n)| \leq d_{n-1}$ and~$a \in W_e^{\sigma\xi \oplus C}$.
	Take the condition $(\sigma\xi, X)$ as the extension.

	\item Case 2: $a \in A_j \setminus W$.
	By definition of~$W$, the collection~$\Ccal$ of colorings~$g : [X]^n \to d_n+1$
	such that for every set~$\xi \in [X]^{<\omega}$
	satisfying~$|g([\xi]^n)| \leq d_{n-1}$, $a \not \in W_e^{\sigma\xi \oplus C}$
	is a non-empty $\Pi^{0,X \oplus C}_1$ class.
	As $\wkl$ admits preservation of $k$ non-c.e.\ definitions,
	there is some coloring~$g \in \Ccal$ such that~$g \oplus X \oplus C$
	preserves non-c.e.\ definitions of the~$A$'s.
	By preservation of $k$ non-c.e.\ definitions of~$\ts^n_{d_{n-1}+1}$,
	there exists an infinite subset~$Y \subseteq X$ such that
	$Y \oplus C$ preserves non-c.e.\ definitions of the~$A$'s
	and~$|g([Y]^n)| \leq d_{n-1}$.
	The condition~$(\sigma, Y)$ forces~$a \not \in W_e^{G \oplus C}$.
\end{itemize}
\end{proof}

Using Lemma~\ref{lem:ts-lemma-set-g} and property (b) of the set~$E$,
we can construct an infinite descending sequence of conditions~$(\epsilon,E) \geq c_0 \geq \dots$
such that for each~$s \in \omega$
\begin{itemize}
	\item[(i)] $\sigma_{s+1} = \sigma_s\xi_s$ with $|\sigma_s| \geq s$ and ~$f([\xi_s]^n) \leq d_{n-1}$
	\item[(ii)] $f(\sigma,\tau) \in I_t$ for each~$t \in (0,n)$, $\sigma \in [\sigma_s]^t$ and~$\tau \in [X]^{n-t}$.
	\item[(iii)] $c_s$ forces~$W_e^{G \oplus C} \neq A_j$ if~$s = \tuple{e,j}$
\end{itemize}
where~$c_s = (\sigma_s, X_s)$. The set~$G = \bigcup_s \sigma_s$
satisfies the desired properties.

\subsubsection{Step (S4) : Construction of the set~$H$}

Finally, we build an infinite set~$H \subseteq G$
such that~$H \oplus C$ preserves non-c.e.\ definitions of the~$A$'s
and~$|f([H]^n)| \leq d_n$.

For each~$i < \omega$, let~$J_i = f([\xi_i]^n)$.
By property (b) of step (S3), $J_i$ is a subset of~$\{0, \dots, d_n\}$
such that $|J_i| \leq d_{n-1}$. For each subset~$J \subseteq \{0, \dots, d_n\}$
of size~$d_{n-1}$, define the set
\[
Z_J = \{ x \in G : (\exists i) x \in \xi_i \wedge f([\xi_i]^n) \subseteq J \}
\]
There exists finitely many such~$J$'s, and the~$Z$'s form a partition of~$G$.
Apply strong preservation of $k$ non-c.e.\ definitions of~$\ts^1_{d_1+1}$
to obtain a finite set~$S$ of $J$'s of such that~$|S| \leq d_1$
and an infinite set~$H \subseteq \bigcup_{J \in S} Z_J \subseteq G$
such that~$H \oplus G \oplus C$ preserves non-c.e.\ definitions of the~$A$'s.

\begin{lemma}
$|f([H]^n)| \leq d_n$
\end{lemma}
\begin{proof}
As~$H \subseteq G$, any~$\sigma \in [H]^n$ can be decomposed into $\rho^\concat \tau$ for some~$\rho \in [\xi_i]^{<\omega}$
and some~$\tau \in [\bigcup_{j \geq i} \xi_j]^{<\omega}$ with~$|\rho| > 0$.
If~$|\tau| = 0$ then $f(\sigma) \in \bigcup_{J \in S} J$  by definition of~$H$.
If~$|\tau| > 0$, then~$f(\sigma) \in I_{|\rho|}$ by property (c) of step (S3).
In any case
\[
f(\sigma) \in (\bigcup_{J \in S} J) \cup (\bigcup_{t \in (0,n)} I_t)
\]
Recall that $|S| \leq d_1$, $|J| = d_{n-1}$ for each~$J \in S$, and~$|I_t| \leq d_t d_{n-t}$ for each~$t \in (0,n)$.
Thus, applying the definition of~$d_n$ from (A3),
$|f([H]^n)| \leq d_1 d_{n-1} + \sum_{0 < t < n} d_t d_{n-t} = d_n$, as desired.
\end{proof}

This completes property (A3) and the proof of Theorem~\ref{thm:ts-arbitrary-preservation}.

\subsection{Free set theorem for tuples and non-c.e.\ definitions}\label{sect:fs-omega-models}

Cholak et al.~\cite{Cholak2001Free} studied the thin set theorem with infinitely many colors as a weakening
of the free set theorem. The forcing notions used by Wang in~\cite{Wang2014Definability}
and by the author in~\cite{PateyCombinatorial} for constructing solutions to free set instances
both involve the thin set theorem for a finite, but arbitrary number of colors.
These constructions may suggest some relation between~$\fs^n$ and~$\ts^n_k$
for arbitrarily large~$k$, but the exact nature of this relation is currently unclear.

In this section, we use the preservation of non-c.e.\ definitions of the thin set theorem
to deduce similar preservation results for the free set theorem,
and thereby separate~$\fs$ from~$\rt^2_2$ over~$\rca$.
More precisely, we prove the following preservation theorem.

\begin{theorem}\label{thm:fs2-preservation-definitions}
For every~$k \in \omega$,
$\fs$ admits strong preservation of $k$ non-c.e.\ definitions.
\end{theorem}

The proof of Theorem~\ref{thm:fs2-preservation-definitions}
follows Corollary~\ref{cor:fs-not-rt22-omega}.
Cholak et al.~\cite{Cholak2001Free}
asked whether any of~$\fs^2$, $\fs^2+\coh$ and
$\fs^2+\wkl$ imply~$\rt^2_2$ and
Hirschfeldt~\cite{Hirschfeldt2015Slicing} asked whether~$\fs^2+\wkl$
implies any of~$\srt^2_2$, $\ads$ or~$\cac$.
We answer all of those questions negatively with the following corollary.

\begin{corollary}\label{cor:fs2-not-imply-sts2n}
For every~$k \geq 2$,
let~$\Phi$ be the conjunction of~$\coh$, $\wkl$, $\rrt^2_2$,
$\pizog$, $\emo$, $\ts^2_{k+1}$ and~$\fs$.
Over~$\rca$, $\Phi$ implies neither~$\sts^2_k$ nor~$\sads$.
\end{corollary}

\begin{corollary}\label{cor:fs-not-rt22-omega}
$\fs$ does not imply~$\rt^2_2$ over~$\omega$-models.
\end{corollary}

The remainder of this section is devoted to the proof of Theorem~\ref{thm:fs2-preservation-definitions}.
The proof is done by induction over the size of the tuples.
The base case of our induction states that $\fs^0$ admits strong preservation of $k$ non-c.e.\ definitions. 
Consider $\fs^0$ as a degenerate case of the free
set theorem, where an instance is a constant~$c$ and a solution to~$c$ is an infinite set~$H$
which does not contain~$c$. Indeed, a function~$f : [\omega]^0 \to \omega$ can be considered as a constant~$c$,
and a set~$H$ is $f$-free if for every~$\varepsilon \in [H]^0$, $f(\varepsilon) \in H \imp f(\varepsilon) \in \varepsilon$.
As~$f(\varepsilon) \not \in \varepsilon$, $f(\varepsilon) = c \not \in H$.
From now on, we will assume that~$\fs^t$ admits strong 
preservation of~$k$ non-c.e.\ definitions for every~$t \in [0,n)$.

We start with a lemma similar to Lemma~\ref{lem:ts-strong-to-weak}.

\begin{lemma}\label{lem:fs-strong-to-weak}
For every~$n \geq 1$ and~$k \geq 2$, if~$\fs^n$ admits strong preservation of~$k$ non-c.e.\ definitions,
then $\fs^{n+1}$ admits preservation of~$k$ non-c.e.\ definitions.
\end{lemma}
\begin{proof}
Fix any set~$C$, $k$ non-$C$-c.e.\ sets~$A_0, \dots, A_{k-1}$ and any $C$-computable
coloring $f : [\omega]^{n+1} \to \omega$.
Consider the uniformly~$C$-computable sequence of sets~$\vec{R}$ defined for each~$\sigma \in [\omega]^n$ and $y \in \omega$ by
\[
R_{\sigma,y} = \{s \in \omega : f(\sigma,s) = y\}
\]
As~$\coh$ admits preservation of~$k$ non-c.e.\ definitions, there exists
some~$\vec{R}$-cohesive set~$G$ such that $G \oplus C$ preserves non-c.e.\ definitions
of the~$A$'s. The cohesive set induces a coloring~$\tilde{f} : [\omega]^n \to \omega$ 
defined for each~$\sigma \in [\omega]^n$ by
\[
\tilde{f}(\sigma) = \cond{
	\lim_{s \in G} f(\sigma,s) & \mbox{ if it exists}\\
	0 & \mbox{ otherwise}\\
}
\]
As~$\fs^n$ admits strong preservation of $k$ non-c.e.\ definitions,
there exists an infinite $\tilde{f}$-free set~$H$ such that
$H \oplus G \oplus C$ preserves non-c.e.\ definitions
of the~$A$'s. In particular,
\[
(\forall \sigma \in [H]^n)(\forall y \in H \setminus \sigma)(\exists b)(\forall s > b)f(\sigma, s) \neq y
\]
$H \oplus G \oplus C$ computes an infinite $f$-free set.
\end{proof}

\subsubsection{Trapped functions}

Although the notion of free set can be defined for every coloring
over tuples of integers, we shall restrict ourselves to a particular kind
of colorings: left trapped functions. The notion of trapped function
was introduced by Wang in~\cite{Wang2014Some} to prove that~$\fs$ does not imply~$\aca$
over~$\omega$-models. It was later reused by the author in~\cite{PateyCombinatorial} to separate~$\fs$
from~$\wwkl$ over~$\omega$-models.

\begin{definition}
A function $f : [\omega]^n \to \omega$ is \emph{left (resp. right) trapped}
if for every $\sigma \in [\omega]^n$, $f(\sigma) \leq \sigma(n-1)$ (resp. $f(\sigma) > \sigma(n-1)$).
\end{definition}

The following lemma is a particular case of a more general statement
proven by the author in~\cite{PateyCombinatorial}.
It follows from the facts that $\fs^n$ for right trapped functions
is strongly uniformly reducible to the diagonally non-computable principle ($\dnr$),
which itself is computably reducible to~$\fs^n$ for left trapped functions.

\begin{lemma}[Patey in~\cite{PateyCombinatorial}]\label{lem:avoidance-fs-trapped-to-untrapped}
For each $k, n \geq 1$, 
if $\fs^n$ for left trapped functions admits (strong) 
preservation of $k$ non-c.e.\ definitions then so does $\fs^k$.
\end{lemma}

It therefore suffices to prove strong preservation
of~$k$ non-c.e.\ definitions for left trapped functions.

\subsubsection{Case of left trapped functions}

In this part, we will prove the following theorem which,
together with Lemma~\ref{lem:avoidance-fs-trapped-to-untrapped}
is sufficient to prove Theorem~\ref{thm:fs2-preservation-definitions} by induction over~$n$.

\begin{theorem}\label{thm:strong-preservation-fs-left-trapped}
For each~$k, n \geq 1$, if~$\fs^t$ admits strong preservation of~$k$ non-c.e.\ definitions
for each~$t \in [0,n)$, then so does~$\fs^n$ for left trapped functions.
\end{theorem}

The proof of Theorem~\ref{thm:strong-preservation-fs-left-trapped}
begins after Lemma~\ref{lem:fs-cohesivity-strong-preservation} and ends after Lemma~\ref{lem:fs2-left-trapped-preserves-2}.
The two following lemmas will ensure that the reservoirs of our forcing conditions
will have good properties, so that the conditions will be extensible.

\begin{lemma}[Patey in~\cite{PateyCombinatorial}]\label{lem:fs-left-trapped-preserves-smaller}
Suppose that $\fs^t$ admits strong preservation of $k$ non-c.e.\ definitions for each $t \in (0,n)$ for some~$k \in \omega$.
Fix a set $C$, some non-$C$-c.e.\ sets~$A_0, \dots, A_{k-1}$, a finite set $F$ and an infinite set $X$ computable in $C$.
For every function $f : [X]^n \to \omega$ there exists an infinite
set $Y \subseteq X$ such that $Y \oplus C$ preserves non-c.e.\ definitions of the~$A$'s and 
$(\forall \sigma \in [F]^t)(\forall \tau \in [Y]^{n-t})f(\sigma, \tau) \not \in Y \setminus \tau$
for each $t \in (0,n)$.
\end{lemma}
\begin{proof}
Fix the finite enumeration $\sigma_0, \dots, \sigma_{m-1}$ of all elements of $[F]^t$ for all $t \in (0,n)$.
We define a finite decreasing sequence of sets~$X = Y_0 \supseteq Y_1 \supseteq \dots \supseteq Y_m$
such that for each~$s < m$
\begin{itemize}
	\item[(a)] none of the $A$'s are $Y_{s+1} \oplus C$-c.e.\
	\item[(b)] $\forall \tau \in [Y_{s+1}]^{n - |\sigma_s|}f(\sigma_s, \tau) \not \in Y_{s+1} \setminus \tau$
\end{itemize}
Given some stage~$s < m$ and some set~$Y_s$,
define the function $f_{\sigma_s} : [Y_s]^{n - |\sigma_s|} \to \omega$ 
by $f_{\sigma_s}(\tau) = f(\sigma_s, \tau)$.
By strong preservation of $k$ non-c.e.\ definitions of $\fs^{n - |\sigma_s|}$,
there exists an infinite set $Y_{s+1} \subseteq Y_s$ satisfying (a) and~(b).
We claim that~$Y_m$ satisfies the properties of the lemma.
Fix some~$\sigma \in [F]^t$ and some~$\tau \in [Y_m]^{n - t}$ for some~$t \in (0,n)$.
There is a stage~$s < m$ such that $\sigma = \sigma_s$. Moreover, $\tau \in [Y_{s+1}]^{n - |\sigma_s|}$, 
so by (b), $f(\sigma_s, \tau) \not \in Y_{s+1} \setminus \tau$, therefore~$f(\sigma, \tau) \not \in Y_m \setminus \tau$, completing the proof.
\end{proof}

\begin{lemma}\label{lem:fs-cohesivity-strong-preservation}
Suppose that $\ts^t_{d_t+1}$ admits strong preservation of~$k$ non-c.e.\ definitions for each $t \in (0,n]$
and $\fs^t$ admits strong preservation of~$k$ non-c.e.\ definitions for each $t \in [0,n)$.
For every function $f : [\omega]^n \to \omega$
and every set $C$ preserving non-c.e.\ definitions of some sets~$A_0, \dots, A_{k-1}$,
there exists an infinite set $X$ such that $X \oplus C$ preserves non-c.e.\ definitions of the~$A$'s 
and for every $\sigma \in [X]^{<\omega}$ such that $0 \leq \card{\sigma} < n$,
$$
(\forall x \in X \setminus \sigma)(\exists b)(\forall \tau \in [X \cap (b, +\infty)]^{n-|\sigma|}) 
  f(\sigma, \tau) \neq x
$$
\end{lemma}
\begin{proof}
Let $X$ be the infinite set constructed in Theorem~\ref{thm:generalized-cohesivity-strong-avoidance}
with $t = n$. For each $s < n$ and $i < d_{n-s}$, let $f_{s,i} : [X]^s \to \omega$ be the function
such that $f_{s,i}(\sigma)$ is the $i$th element of
$\set{x : (\forall b)(\exists \tau \in [X \cap (b, +\infty)]^{n-s}) f(\sigma, \tau) = x}$
if it exists, and 0 otherwise.
Define a finite sequence $X \supseteq X_0 \supseteq \dots \supseteq X_{n-1}$ such that for each $s  < n$
\begin{itemize}
  \item[1.] $X_s \oplus C$ preserves non-c.e.\ definitions of the~$A$'s 
  \item[2.] $X_s$ is $f_{s,i}$-free for each $i < d_{n-s}$
\end{itemize}
We claim that $X_{n-1}$ is the desired set. Fix $s < n$ and take any $\sigma \in [X_{n-1}]^s$
and any $x \in X_{n-1} \setminus \sigma$. 
If $(\forall b)(\exists \tau \in [X \cap (b, +\infty)]^{n - s})f(\sigma, \tau) = x$,
then by choice of $X$, there exists an $i < d_{n - s}$ such that
$f_{s, i}(\sigma) = x$, contradicting $f_{s,i}$-freeness of $X_{n-1}$.
So $(\exists b)(\forall \tau \in [X \cap (b, +\infty)]^{n - s})f(\sigma, \tau) \neq x$.
\end{proof}

\begin{proof}[Proof of Theorem~\ref{thm:strong-preservation-fs-left-trapped}]
Fix~$k \geq 2$, some set~$C$, some non-$C$-c.e.\ sets~$A_0, \dots, A_{k-1}$
and a left trapped coloring~$f : [\omega]^n \to \omega$.
We will construct an infinite $f$-free set~$H$ such that
none of the $A$'s is $H \oplus C$-c.e.
Our forcing conditions are Mathias conditions $(F, X)$ such that
\begin{itemize}
  \item[(a)] $X \oplus C$ preserves non-c.e.\ definitions of the~$A$'s
  \item[(b)] $(\forall \sigma \in [F \cup X]^n) f(\sigma) \not \in F \setminus \sigma$
  \item[(c)] $(\forall \sigma \in [F \cup X]^t)(\forall x \in (F \cup X) \setminus \sigma)
  (\exists b)(\forall \tau \in [(F \cup X) \cap (b, +\infty)]^{n-t})\\
  f(\sigma, \tau) \neq x$ for each $t \in [0, n)$.
  \item[(d)] $(\forall \sigma \in [F]^t)(\forall \tau \in [X]^{n-t})
  f(\sigma, \tau) \not \in X \setminus \tau$ for each $t \in (0,n)$
\end{itemize}
Properties (c) and (d) will be obtained by 
Lemma~\ref{lem:fs-cohesivity-strong-preservation} 
and Lemma~\ref{lem:fs-left-trapped-preserves-smaller} and are present to maintain the property (b)
over extensions. A set $G$ \emph{satisfies} a condition $(F, X)$ if it is $f$-free and
satisfies the Mathias condition $(F, X)$.
Our initial condition is $(\emptyset, Y)$ where $Y$ is obtained by Lemma~\ref{lem:fs-cohesivity-strong-preservation}.

\begin{lemma}\label{lem:fs2-left-trapped-preserves-1}
For every condition $(F, X)$ there exists an extension
$(H, Y)$ such that $|H| > |F|$.
\end{lemma}
\begin{proof}
Choose an $x \in X$ such that 
$(\forall \sigma \in [F]^n)f(\sigma) \neq x$ and set $H = F \cup \{x\}$.
By property (c) of $(F, X)$, there exists a $b$
such that
\[
(\forall \sigma \in [F]^{t})
(\forall \tau \in [X \cap (b, +\infty)]^{n-t})f(\sigma, \tau) \neq \{x\} \setminus \sigma
\]
for each $t \in [0, n]$.
By Lemma~\ref{lem:fs-left-trapped-preserves-smaller},
there exists an infinite set $Y \subseteq X \setminus [0, b]$ such that $Y \oplus C$ preserves
non-c.e.\ definitions of the~$A$'s and property (d) is satisfied for $(H, Y)$.
We claim that $(H, Y)$ is a valid condition.
Properties (a), (c) and (d) trivially hold. We now check property (b).
By property (b) of $(F, X)$, we only need to check that
$(\forall \sigma \in [F \cup Y]^k)f(\sigma) \neq x$.
This follows from our choice of~$b$.
\end{proof}

\begin{lemma}\label{lem:fs2-left-trapped-preserves-2}
For every condition $(F, X)$, every~$e \in \omega$ and~$j < k$,
there exists an extension $(H, Y)$
forcing $W_e^{G \oplus C} \neq A_j$. 
\end{lemma}
\begin{proof}
By removing finitely many elements of $X$, we can assume that
$(\forall \sigma \in [F]^n)f(\sigma) \not \in X$.
For each~$a \in \omega$, let~$\Ccal_a$ be the~$\Pi^{0,X \oplus C}_1$ class
of left trapped functions~$g : [X]^n \to \omega$ such that for every $g$-free
set~$E \subset X$, $a \not \in W_e^{(F \cup E) \oplus C}$.
Also define $W = \{ a \in \omega : \Ccal_a = \emptyset \}$.
The set~$W$ is $X \oplus C$-c.e.\ but~$A_j$ is not $X \oplus C$-c.e., therefore $W \neq A_j$.
Let~$a \in W \Delta A_j$. We have two cases:
\begin{itemize}
	\item Case 1: $a \in W \setminus A_j$. As~$f \not \in \Ccal_a$, 
	there exists a finite $f$-free set~$E$ such that~$a \in W^{(F \cup E) \oplus C}_e$.
	Set~$H = F \cup E$.
	By property (c) of $(F, X)$, there exists a $b$ such that
	\[
	(\forall \sigma \in [H]^t)(\forall x \in H)
	(\forall \tau \in [X \cap (b, +\infty)]^{n-t}) f(\sigma, \tau) \neq \{x\} \setminus \sigma
	\]
	for each~$t \in [0, n)$.
	By Lemma~\ref{lem:fs-left-trapped-preserves-smaller},
	there exists an infinite set $Y \subseteq X \cap (b, +\infty)$ such that 
	$Y \oplus C$ preserves non-c.e.\ definitions of the~$A$'s
	and property (d) is satisfied for $(H, Y)$.
	We claim that $(H, Y)$ is a valid condition.

	Properties (a), (c) and (d) trivially hold. We now check property (b).
	By our choice of $b$, we only need to check that $(\forall \sigma \in [H]^n)
	f(\sigma) \not \in H \setminus \sigma$.
	By property (b) of $(F, X)$, it suffices to check that
	$(\forall \sigma \in [H]^n)f(\sigma) \not \in E \setminus \sigma$.
	By property (d) of $(F, X)$, and our initial assumption on~$X$, we only need to check that
	$(\forall \sigma \in [E]^n)f(\sigma) \not \in E \setminus \sigma$,
	which is exactly $f$-freeness of $E$.

	\item Case 2: $a \in A_j \setminus W$. By definition of~$W$, $\Ccal_a \neq \emptyset$.
	As~$\wkl$ admits preservation of~$k$ non-c.e.\ definitions,
	there exists a left trapped functions~$g \in \Ccal_a$
	such that $g \oplus X \oplus C$ preserves non-c.e.\ definitions of the~$A$'s.
	As~$\fs^n$ admits preservation of~$k$ non-c.e.\ definitions,
	there exists some infinite $g$-free set~$Y \subseteq X$ such that~$Y \oplus C$
	preserves non-c.e.\ definitions of the~$A$'s.
	The condition $(F, Y)$ forces $a \not \in W_e^{G \oplus C}$ and therefore~$W_e^{G \oplus C} \neq A_j$.
\end{itemize}
\end{proof}

Let~$Y$ be the set constructed in Lemma~\ref{lem:fs-cohesivity-strong-preservation}.
Using Lemma~\ref{lem:fs2-left-trapped-preserves-1}
and Lemma~\ref{lem:fs2-left-trapped-preserves-2}, 
we can define an infinite decreasing sequence of conditions~
$(\emptyset, Y) \geq c_0 \geq \dots$ such that for every~$s \in \omega$
\begin{itemize}
	\item[(i)] $|F_s| \geq s$
	\item[(ii)] $c_s$ forces~$W_e^{G \oplus C} \neq A_j$ if~$s = \tuple{e,j}$
\end{itemize}
where~$c_s = (F_s, X_s)$. Let~$G = \bigcup_s F_s$. By (i), $G$ is infinite
and by~(ii), none of the~$A$'s are~$G \oplus C$-c.e.
This completes the proof of Theorem~\ref{thm:strong-preservation-fs-left-trapped}.
\end{proof}

\section{Conclusion}

In this last section, we recall some existing open questions about the free set
and thin set theorems, and state some new ones.
Cholak et al.~\cite{Cholak2001Free} asked the following question
which remains open.

\begin{question}
Does~$\ts^2$ imply~$\fs^2$ over~$\rca$?
\end{question}

We ask a related question motivated by the fact that
the proof of cone avoidance of~$\fs$ by Wang~\cite{Wang2014Some}
and the preservation of $k$ non-c.e.\ definitions
of~$\fs^2$ in section~\ref{sect:fs-omega-models} both use~$\ts^2_k$ for any~$k$
to construct a solution to an instance of~$\fs^2$.
We know by Corollary~\ref{cor:fs2-not-imply-sts2n} that $\fs^2$ does not imply~$\ts^2_k$ over~$\rca$
for any~$k$, but the reverse implication is still open.

\begin{question}
Does $\ts^2_3$ imply~$\fs^2$ over~$\rca$?
\end{question}

Cholak et al.~\cite{Cholak2001Free} also asked whether~$\fs^2+\cac$ implies~$\rt^2_2$
over~$\rca$. Using the equivalence between~$\rt^2_2$ and~$\emo+\ads$ 
proven by Bovykin and Weiermann~\cite{Bovykin2005strength},
we ask the following related questions.

\begin{question}
Does any of~$\fs^2$, $\ts^2$ and $\ts^2_3$ imply~$\emo$ over~$\rca$?
\end{question}

\vspace{0.5cm}

\noindent \textbf{Acknowledgements}. 
The author is thankful to his PhD advisor Laurent Bienvenu 
and to Wei Wang for useful comments and discussions.
The author is funded by the John Templeton Foundation (`Structure and Randomness in the Theory of Computation' project). 
The opinions expressed in this publication are those of the author(s) and do not necessarily reflect the views of the John Templeton Foundation.


\vspace{0.5cm}


\begin{thebibliography}{10}

\bibitem{Avigad2012Algorithmic}
Jeremy Avigad, Edward~T. Dean, and Jason Rute.
\newblock {Algorithmic randomness, reverse mathematics, and the dominated
  convergence theorem}.
\newblock {\em Annals of Pure and Applied Logic}, 163(12):1854--1864, 2012.

\bibitem{Bienvenu2015logical}
Laurent Bienvenu, Ludovic Patey, and Paul Shafer.
\newblock {On the logical strengths of partial solutions to mathematical
  problems}.
\newblock Submitted. Available at \url{http://arxiv.org/abs/1411.5874}, 2015.

\bibitem{Bovykin2005strength}
Andrey Bovykin and Andreas Weiermann.
\newblock The strength of infinitary {R}amseyan principles can be accessed by
  their densities.
\newblock {\em Annals of Pure and Applied Logic}, page~4, 2005.
\newblock To appear.

\bibitem{Brattka2015Uniform}
Vasco Brattka and Tahina Rakotoniaina.
\newblock On the uniform computational content of {R}amsey's theorem.
\newblock Available at \url{http://arxiv.org/abs/1508.00471}., 2015.

\bibitem{Cholak2001Free}
Peter~A. Cholak, Mariagnese Giusto, Jeffry~L. Hirst, and Carl~G. Jockusch~Jr.
\newblock Free sets and reverse mathematics.
\newblock {\em Reverse mathematics}, 21:104--119, 2001.

\bibitem{Cholak2001strength}
Peter~A. Cholak, Carl~G. Jockusch, and Theodore~A. Slaman.
\newblock {On the strength of Ramsey's theorem for pairs}.
\newblock {\em Journal of Symbolic Logic}, 66(01):1--55, 2001.

\bibitem{Chong2010role}
C.~Chong, Steffen Lempp, and Yue Yang.
\newblock {On the role of the collection principle for $\Sigma^0_2$-formulas in
  second-order reverse mathematics}.
\newblock {\em Proceedings of the American Mathematical Society},
  138(3):1093--1100, 2010.

\bibitem{Chong2014metamathematics}
Chitat Chong, Theodore Slaman, and Yue Yang.
\newblock {The metamathematics of stable {R}amsey’s theorem for pairs}.
\newblock {\em Journal of the American Mathematical Society}, 27(3):863--892,
  2014.

\bibitem{Dorais2016uniform}
Fran{\c{c}}ois~G. Dorais, Damir~D. Dzhafarov, Jeffry~L. Hirst, Joseph~R.
  Mileti, and Paul Shafer.
\newblock On uniform relationships between combinatorial problems.
\newblock {\em Trans. Amer. Math. Soc.}, 368(2):1321--1359, 2016.

\bibitem{DzhafarovStrong}
Damir~D. Dzhafarov.
\newblock Strong reductions between combinatorial principles.
\newblock In preparation.

\bibitem{Dzhafarov2012Cohesive}
Damir~D. Dzhafarov.
\newblock Cohesive avoidance and arithmetical sets.
\newblock {\em arXiv preprint arXiv:1212.0828}, 2012.

\bibitem{Dzhafarov2014Cohesive}
Damir~D. Dzhafarov.
\newblock Cohesive avoidance and strong reductions.
\newblock {\em Proceedings of the American Mathematical Society},
  143(2):869--876, 2014.

\bibitem{Friedberg1957criterion}
Richard Friedberg.
\newblock A criterion for completeness of degrees of unsolvability.
\newblock {\em J. Symb. Logic}, 22:159--160, 1957.

\bibitem{FriedmanFom53free}
Harvey~M. Friedman.
\newblock Fom:53:free sets and reverse math and fom:54:recursion theory and
  dynamics.
\newblock Available at \url{https://www.cs.nyu.edu/pipermail/fom/}.

\bibitem{Friedman2013Boolean}
Harvey~M. Friedman.
\newblock {\em Boolean Relation Theory and Incompleteness}.
\newblock {Lecture Notes in Logic}, 2013.
\newblock to appear. Available at
  http://www.math.ohio-state.edu/

\bibitem{Greenberg2009Lowness}
Noam Greenberg and Joseph~S. Miller.
\newblock {Lowness for Kurtz randomness}.
\newblock {\em Journal of Symbolic Logic}, 74(2):665--678, 2009.

\bibitem{Hirschfeldt2015Slicing}
Denis~R. Hirschfeldt.
\newblock {\em Slicing the truth}, volume~28 of {\em Lecture Notes Series.
  Institute for Mathematical Sciences. National University of Singapore}.
\newblock World Scientific Publishing Co. Pte. Ltd., Hackensack, NJ, 2015.
\newblock On the computable and reverse mathematics of combinatorial
  principles, Edited and with a foreword by Chitat Chong, Qi Feng, Theodore A.
  Slaman, W. Hugh Woodin and Yue Yang.

\bibitem{Hirschfeldtnotions}
Denis~R. Hirschfeldt and Carl~G. Jockusch.
\newblock {On notions of computability theoretic reduction between $\Pi^1_2$
  principles}.
\newblock {\em To appear}.

\bibitem{Hirschfeldt2007Combinatorial}
Denis~R. Hirschfeldt and Richard~A. Shore.
\newblock {Combinatorial principles weaker than {R}amsey's theorem for pairs}.
\newblock {\em Journal of Symbolic Logic}, 72(1):171--206, 2007.

\bibitem{Jockusch1972Ramseys}
Carl~G. Jockusch.
\newblock Ramsey's theorem and recursion theory.
\newblock {\em Journal of Symbolic Logic}, 37(2):268--280, 1972.

\bibitem{JockuschJr1980Degrees}
Carl~G. Jockusch.
\newblock Degrees of generic sets.
\newblock {\em Recursion Theory: its generalizations and applications}, pages
  110--139, 1980.

\bibitem{Jockusch1991}
Carl~G. Jockusch, A~Lewis, and Jeffrey~B. Remmel.
\newblock {$\Pi^0_1$-classes and Rado's selection principle}.
\newblock {\em Journal of Symbolic Logic}, 56(02):684--693, 1991.

\bibitem{Jockusch197201}
Carl~G. Jockusch and Robert~I. Soare.
\newblock {$\Pi^0_1$} classes and degrees of theories.
\newblock {\em Transactions of the American Mathematical Society}, 173:33--56,
  1972.

\bibitem{Jockusch1993cohesive}
Carl~G. Jockusch and Frank Stephan.
\newblock A cohesive set which is not high.
\newblock {\em Mathematical Logic Quarterly}, 39(1):515--530, 1993.

\bibitem{Kautz1991Degrees}
Steven~M. Kautz.
\newblock {\em Degrees of random sets}.
\newblock PhD thesis, Citeseer, 1991.

\bibitem{Kautz1998improved}
Steven~M. Kautz.
\newblock An improved zero-one law for algorithmically random sequences.
\newblock {\em Theoretical computer science}, 191(1):185--192, 1998.

\bibitem{Khan2014Forcing}
Mushfeq Khan and Joseph~S. Miller.
\newblock Forcing with bushy trees.
\newblock preprint, 2014.

\bibitem{Kjos-Hanssen2009Infinite}
Bj{\o}rn Kjos-Hanssen.
\newblock Infinite subsets of random sets of integers.
\newblock {\em Mathematics Research Letters}, 16:103--110, 2009.

\bibitem{Lerman2013Separating}
Manuel Lerman, Reed Solomon, and Henry Towsner.
\newblock {Separating principles below {R}amsey's theorem for pairs}.
\newblock {\em Journal of Mathematical Logic}, 13(02):1350007, 2013.

\bibitem{Liu2012RT22}
Lu~Liu.
\newblock {RT$^2_2$ does not imply WKL$_0$}.
\newblock {\em Journal of Symbolic Logic}, 77(2):609--620, 2012.

\bibitem{Mileti2004Partition}
Joseph~Roy Mileti.
\newblock {\em Partition theorems and computability theory}.
\newblock ProQuest LLC, Ann Arbor, MI, 2004.
\newblock Thesis (Ph.D.)--University of Illinois at Urbana-Champaign.

\bibitem{Miller1968degrees}
Webb Miller and Donald~A. Martin.
\newblock The degrees of hyperimmune sets.
\newblock {\em Mathematical Logic Quarterly}, 14(7-12):159--166, 1968.

\bibitem{Montalban2011Open}
Antonio Montalb{\'a}n.
\newblock Open questions in reverse mathematics.
\newblock {\em Bulletin of Symbolic Logic}, 17(03):431--454, 2011.

\bibitem{PateyCombinatorial}
Ludovic Patey.
\newblock Combinatorial weaknesses of {R}amseyan principles.
\newblock In preparation. Available at
  \url{http://ludovicpatey.com/media/research/combinatorial-weaknesses-draft.pdf},
  2015.

\bibitem{Patey2015Degrees}
Ludovic Patey.
\newblock Degrees bounding principles and universal instances in reverse
  mathematics.
\newblock {\em Annals of Pure and Applied Logic}, 166(11):1165--1185, 2015.

\bibitem{Patey2015Iterative}
Ludovic Patey.
\newblock Iterative forcing and hyperimmunity in reverse mathematics.
\newblock In Arnold Beckmann, Victor Mitrana, and Mariya Soskova, editors, {\em
  CiE. Evolving Computability}, volume 9136 of {\em Lecture Notes in Computer
  Science}, pages 291--301. Springer International Publishing, 2015.

\bibitem{Patey2016Controlling}
Ludovic Patey.
\newblock Controlling iterated jumps of solutions to combinatorial problems.
\newblock {\em Computability}, 2016.
\newblock To appear. Available at \url{http://arxiv.org/abs/1509.05340}.

\bibitem{Rakotoniaina2015Computational}
Tahina Rakotoniaina.
\newblock {\em The Computational Strength of {R}amsey’s Theorem}.
\newblock PhD thesis, University of Cape Town, 2015.
\newblock to appear.

\bibitem{Rosenstein1982Linear}
Joseph~G. Rosenstein.
\newblock {\em Linear orderings}, volume~98 of {\em Pure and Applied
  Mathematics}.
\newblock Academic Press, Inc. [Harcourt Brace Jovanovich, Publishers], New
  York-London, 1982.

\bibitem{Seetapun1995strength}
David Seetapun and Theodore~A. Slaman.
\newblock {On the strength of {R}amsey's theorem}.
\newblock {\em Notre Dame Journal of Formal Logic}, 36(4):570--582, 1995.

\bibitem{Shoenfield1959degrees}
Joseph~R. Shoenfield.
\newblock On degrees of unsolvability.
\newblock {\em Annals of Mathematics}, 69(03):644--653, May 1959.

\bibitem{Simpson2007extension}
Stephen~G. Simpson.
\newblock An extension of the recursively enumerable {T}uring degrees.
\newblock {\em Journal of the London Mathematical Society}, 75(2):287--297,
  2007.

\bibitem{Simpson2009Subsystems}
Stephen~G. Simpson.
\newblock {\em {Subsystems of Second Order Arithmetic}}.
\newblock Cambridge University Press, 2009.

\bibitem{Stephan2006Martin}
Frank Stephan.
\newblock Martin-{L}\"of random and {PA}-complete sets.
\newblock In {\em Logic {C}olloquium '02}, volume~27 of {\em Lect. Notes Log.},
  pages 342--348. Assoc. Symbol. Logic, La Jolla, CA, 2006.

\bibitem{VanLambalgen1990axiomatization}
Michiel Van~Lambalgen.
\newblock The axiomatization of randomness.
\newblock {\em Journal of Symbolic Logic}, 55(03):1143--1167, 1990.

\bibitem{Wang2013Omitting}
Wei Wang.
\newblock Omitting cohesive sets.
\newblock {\em arXiv preprint arXiv:1309.5428}, 2013.

\bibitem{Wang2014Definability}
Wei Wang.
\newblock The definability strength of combinatorial principles, 2014.
\newblock To appear. Available at http://arxiv.org/abs/1408.1465.

\bibitem{Wang2014Some}
Wei Wang.
\newblock Some logically weak {R}amseyan theorems.
\newblock {\em Advances in Mathematics}, 261:1--25, 2014.

\end{thebibliography}
\end{document}